\documentclass[a4paper,10pt]{article}
\usepackage{authblk}
\usepackage[english]{babel}
\usepackage{amsfonts,amsthm,amssymb}
\usepackage{cite}
\usepackage{enumerate}
\usepackage[hmargin={28mm,28mm},vmargin={30mm,35mm}]{geometry}
\usepackage{relsize}
\usepackage{xcolor}
\usepackage[ocgcolorlinks,allcolors={blue}]{hyperref}
\usepackage{nicefrac}
\usepackage{mathtools}
\usepackage{bm}
\usepackage{bbm}
\usepackage{mathabx}
\usepackage{setspace}
\usepackage{subfig}
\usepackage{floatrow}
\usepackage{color}
\usepackage[applemac]{inputenc}
\usepackage[T1]{fontenc}
\usepackage{url}
\usepackage{upgreek}
\usepackage{paralist}
\usepackage{pmat}

\usepackage{parskip}

\begingroup
    \makeatletter
    \@for\theoremstyle:=definition,remark,plain\do{%
        \expandafter\g@addto@macro\csname th@\theoremstyle\endcsname{%
            \addtolength\thm@preskip\parskip
            }%
        }
\endgroup

\usepackage{diagbox}
\usepackage{booktabs}

\makeatletter
\def\thm@space@setup{%
  \thm@preskip=\parskip \thm@postskip=0pt
}
\makeatother

\newcommand{\arxiv}[1]{\href{https://arxiv.org/abs/#1}{arXiv:~#1}}

\newcommand{\ul}{\underline}
\newcommand{\ulr}[1]{\underline{\mr #1}}

\newcommand{\pT}{p^{k+2}_{T}}
\newcommand{\ph}{p^{k+2}_{h}}
\newcommand{\tbf}{\textbf}

\newcommand{\wc}{\widecheck}

\newcommand{\bb}{\mathbb}

\newcommand{\pphi}{\varphi}
\newcommand{\nf}{\nicefrac}
\newcommand{\cc}{\mathcal}
\newcommand{\de}{\partial}

\newcommand{\wh}{\widehat}

\newcommand{\grad}  {\bm\nabla}
\renewcommand{\ker} {\mbox{\rm ker}\,}
\newcommand{\divb} {\mbox{\textbf{div}}}		
\renewcommand{\div} {\mbox{\rm{div}}\,}		

\newcommand{\tx}{\text}
\renewcommand{\bf}{\mathbf}
\newcommand{\mr}{\mathrm}

\DeclareMathOperator{\card}{card}
\newcommand{\jump}[2][F]{[#2]_{#1}}

\DeclareRobustCommand{\vect}[1]{\bm{#1}}
\newtheorem{theorem}{Theorem}
\newtheorem{proposition}[theorem]{Proposition}
\newtheorem{lemma}[theorem]{Lemma}

\theoremstyle{remark}
\newtheorem{remark}[theorem]{Remark}
\theoremstyle{definition}

\newcommand{\sU}[1][h]{\mathsf{U}_{#1}}
\newcommand{\sV}[1][h]{\mathsf{V}_{#1}}
\newcommand{\sA}[1][h]{\mathsf{A}_{#1}}
\newcommand{\sB}[1][h]{\mathsf{B}_{#1}}

\newcommand{\trans}{^\intercal}

\newcommand{\email}[1]{\href{mailto:#1}{#1}}

\begin{document}

\title{A Hybrid High-Order method for Kirchhoff--Love plate bending problems\footnote{The work of the first author was supported by Agence Nationale de la Recherche projects HHOMM (ANR-15-CE40-0005), along with the second author, and ARAMIS (ANR-12-BS01-0021), along with the third and fourth authors; also, it was partially supported by SIR Research Grant no.~RBSI14VTOS funded by MIUR -- Italian Ministry of Education, Universities and Research.}}
\author[1,2]{Francesco Bonaldi\footnote{Corresponding author, \email{francesco.bonaldi@polimi.it}}}
\affil[1]{MOX, Dipartimento di Matematica, Politecnico di Milano, Milan, Italy}

\author[2]{Daniele A. Di Pietro\footnote{\email{daniele.di-pietro@umontpellier.fr}}}
\affil[2]{Institut Montpelliérain Alexander Grothendieck, CNRS, Univ. Montpellier}

\author[3]{Giuseppe Geymonat\footnote{\email{giuseppe.geymonat@polytechnique.edu}}}
\affil[3]{LMS, Ecole Polytechnique, CNRS, Universit\'{e} Paris-Saclay, 91128 Palaiseau, France}

\author[2]{Fran\c{c}oise Krasucki\footnote{\email{francoise.krasucki@umontpellier.fr}}}

\maketitle

\begin{abstract}
  \noindent
  We present a novel Hybrid High-Order (HHO) discretization of fourth-order elliptic problems arising from the mechanical modeling of the bending behavior of Kirchhoff--Love plates, including the biharmonic equation as a particular case.
  The proposed HHO method supports arbitrary approximation orders on general polygonal meshes, and reproduces the key mechanical equilibrium relations locally inside each element.
  When polynomials of degree $k\ge 1$ are used as unknowns, we prove convergence in $h^{k+1}$ (with $h$ denoting, as usual, the meshsize) in an energy-like norm.
  A key ingredient in the proof are novel approximation results for the {energy projector} on local polynomial spaces.
  Under biharmonic regularity assumptions, a sharp estimate in $h^{k+3}$ is also derived for the $L^2$-norm of the error on the deflection.
  The theoretical results are supported by numerical experiments, which additionally show the robustness of the method with respect to the choice of the stabilization.
\end{abstract}



\section{Introduction}
As remarked by O.~C.~Zienkiewicz \cite{zienk1}, ``one of the early requirements of the Finite Element (FE)
approximation was the choice of shape functions which did not lead to infinite
strains on element interfaces and which therefore preserved a necessary degree
of continuity''. This requirement (also called of {conformity}) appeared easy to satisfy for simple self-adjoint problems
governed by second-order equations, where $C^{0}$-continuity at interfaces is enough. The situation is different as far as it concerns the knowledge, essential in structural engineering, of the bending of plates, whose numerical treatment has always been a goal of FE computations. Since thin plate bending in the Kirchhoff--Love approximation is governed by a fourth-order equation, $C^{1}$-continuity has to be introduced (and the continuity of both the function and of its normal derivative assured at interfaces). This was difficult to achieve and computationally expensive in the classical FE framework, see e.g.~Zienkiewicz \cite{zienk} for a first engineering-oriented discussion and Ciarlet \cite{ciarlet-num} for a mathematically-oriented one. In order to relax such $C^{1}$-continuity condition, many non-conforming, mixed, hybrid plates elements have been studied and tested all over the last fifty years, and the literature on this subject is very broad; a minimal and by far non-exhaustive sample includes  the seminal paper by Lascaux--Lesaint \cite{lascaux}, as well as the classical works of {Amara--Capatina--Chatti \cite{Amara.Capatina.ea:02} (based on a decomposition of the constraints imposed on the bending moments by applying twice the Tartar lemma and using the symmetry of the tensor)}, Bathe \cite{bathe}, Boffi--Brezzi--Fortin \cite{boffi}, Brenner \cite{brenner}, Brenner--Scott \cite{brennerscott}, Brezzi--Fortin \cite{brezzifortin}, Ciarlet\cite{ciarlet-num}, {Comodi \cite{Comodi:89}}, Hughes \cite{hug}, {Johnson \cite{Johnson:73}}; see also references therein.
More recent nonconforming methods which have similarities (and differences) with the one presented here include the Hybridizable Discontinuous Galerkin method~\cite{Cockburn.Dong.ea:09} of Cockburn--Dong--Guzm\'an and the Weak Galerkin method~\cite{Lin.Wang.ea:14} of Lin--Wang--Ye;
see also \cite{Cockburn.Gopalakrishnan.ea:09} concerning the passage from Discontinuous Galerkin to hybrid methods.
We also cite here {the mixed method of Behrens--Guzm\'an \cite{Behrens.Guzman:11} based on a system of first-order equations, and the HHO method of} \cite{Chave.Di-Pietro.ea:16}, where the fourth-order operator in the Cahn--Hilliard equations is treated as a system of second-order operators.

A recent approach to the construction of FE spaces with $C^1$-regularity, on the other hand, has been developed in the context of the Virtual Element Method (VEM)~\cite{Antonietti.Beirao-da-Veiga.ea:15,Beirao-da-Veiga.Manzini:14,brezziseoul}.
Here, global continuity requirements are enforced by renouncing an explicit expression of the basis functions at each point, and local contributions are built using computable projections thereof (a stabilization term therefore has to be added).
We refer the reader to~\cite{Brezzi.Marini:13} \cite{chinosidona} for an application of $C^1$-conforming virtual spaces to plate-bending problems similar to the ones considered here.
Nonconforming versions of the VEM have also been developed for fourth order operators, see, e.g., the very recent contributions by Antonietti--Manzini--Verani \cite{antonietti} (including nodal unknowns) and Zhao--Chen--Zhang \cite{zhao} (with $C^0$-continuous virtual functions).

The Kirchhoff--Love plate bending model problem considered in this work reads
\begin{subequations}\label{static_hho}
  \begin{alignat}{2}
    -\div\divb \bm M &= f &\qquad&\tx{ in }\Omega, \label{eq:pde} \\
    u &= 0 &\qquad&\tx{ on }\de\Omega, \label{eq:bc1} \\
    \de_{\bm n} u &= 0 &\qquad&\tx{ on }\de\Omega, \label{eq:bc2}
  \end{alignat}
\end{subequations}
where $\Omega\subset \bb R^2$ denotes a two-dimensional bounded and connected polygonal domain, representing
the middle surface of a plate in its reference configuration, and the divergence operator is denoted by div or $\mathbf{div}$, as to whether it acts on vector- or tensor-valued fields, respectively.
In \eqref{eq:pde}, $f$ represents a surface load orthogonal to the plane of the plate, and $\bm M$ is the \emph{moment tensor}, a second-order symmetric tensor field related to the scalar unknown $u$, the \emph{deflection} of the plate, by the constitutive law
$$
\bm M= - \bb A\grad^2 u,
$$
where $\bb A$ is a fourth-order, symmetric and uniformly elliptic tensor field, and $-\grad^2 u$ is referred to as the \emph{curvature tensor}.
For the sake of simplicity, we assume in what follows that $\bb A$ is piecewise constant on a finite polygonal partition \begin{equation}\label{eq:POmega}{P_\Omega=\{\Omega_i\,:\,i\in I\}}\end{equation} of $\Omega$, and that $f\in L^2(\Omega)$.
Variational formulations are classical for problem \eqref{static_hho}.
For $X\subset\overline{\Omega}$, we denote by $(\cdot,\cdot)_X$ the scalar product in $L^2(X)$, $L^2(X)^2$ or
$L^2(X)^{2\times2}$, depending on the context, and by $\|{\cdot}\|_X$ the associated norm; we omit the subscript $X$ whenever $X=\Omega$.
The \textit{primal} variational formulation of problem \eqref{static_hho} reads:
Find $u\in H^2_0(\Omega)$ such that
\begin{equation}\label{varform_hho}
  (\bb A\grad^2 u, \grad^2 v) \eqcolon a(u,v) = (f,v)\qquad\forall v\in H^2_0(\Omega).
\end{equation}
Owing to the Lax--Milgram Lemma, problem \eqref{varform_hho} is well-posed (see, e.g., \cite{necas,brezzifortin}).

In this work, we propose and analyze a novel Hybrid High-Order (HHO) method for the approximation of problem \eqref{varform_hho} which sits at the far end of the spectrum of nonconforming methods, since the underlying space does not even embed $C^0$-continuity.
HHO methods, introduced in \cite{Di-Pietro.Ern:15} in the context of quasi-incompressible linear elasticity, are a class of new-generation discretization methods for partial differential equations with several advantageous features.
The most relevant in the context of plate bending problems are:
\begin{inparaenum}[(i)]
\item the support of arbitrary approximation orders on general polygonal meshes;
\item the reproduction of key continuous properties (such as, e.g., local equilibrium relations) at the discrete level;
\item {competitive} computational cost thanks to static condensation and compact stencil.
\end{inparaenum}
We refer the reader to \cite{Di-Pietro.Tittarelli:17} for an introduction covering the salient aspects of HHO methods for linear and nonlinear problems.

The HHO method for problem \eqref{varform_hho} is formulated in terms of discrete unknowns defined on mesh faces and elements (whence the term \emph{hybrid}), and such unknowns are polynomials of arbitrary degree $k \ge 1$ (whence the expression \emph{high-order}).
The construction is conceived so that only face-based unknowns are globally coupled, whereas element-based unknowns can be eliminated by static condensation; see Remark \ref{rem:implementation} below for further details.
Element-based unknowns play the role of the deflection $u$ inside elements, whereas face unknowns play the role of the traces of $u$ and of its gradient on faces.
From these unknowns, a reconstruction of the deflection of degree $(k+2)$ is obtained by solving a local problem inside each element.
This reconstruction is conceived so that, composed with a local reduction map, it coincides with the {local energy projector} and, as such, has optimal approximation properties in the space of polynomials of total degree (up to) $(k+2)$; see Theorem \ref{thm:err.est} below, whose proof hinges on the recent results of \cite{Di-Pietro.Droniou:17}.
The high-order deflection reconstruction is used to formulate a local contribution, which includes a carefully tailored stabilization term.
The role of the latter is to ensure coercivity with respect to a $H^2$-like seminorm while, at the same time, preserving the approximation properties of the local deflection reconstruction.

An extensive convergence analysis of the method is carried out.
Specifically, in Theorem~\ref{thm:err.est} below we prove convergence in $h^{k+1}$ (with $h$ denoting, as usual, the meshsize) in an energy-like norm and, in Theorem \ref{L2:err.est} below, a sharp estimate in $h^{k+3}$ for the $L^2$-norm under biharmonic regularity assumptions.
The latter result highlights a salient feature of HHO methods, namely the fact that, by construction, element-based unknowns superconverge to the $L^2$-orthogonal projection of the exact solution on general meshes.
As this happens by design (i.e., this behavior is not serendipitous), this phenomenon is henceforth referred to as \emph{supercloseness} rather than \emph{superconvergence}.
We also show that the method satisfies locally inside each element a discrete version of the principle of virtual work with moments and shear forces obeying a law of action and reaction.
The performance of the method is showcased on numerical examples, including a study of the robustness with respect to the choice of the stabilization.

The rest of the paper is organized as follows.
In Section~\ref{sec:discset} we introduce the discrete setting: regularity for polygonal meshes, basic results thereon, and local projectors.
A novel general result contained in this section is Theorem \ref{thm:approx.biell}, where optimal approximation properties for the {local energy projector} on local polynomial spaces are studied. The proof of this theorem is given in Section \ref{sec:approx.biell}.
In Section \ref{sec:hho} we introduce the HHO method, state the main results corresponding to Theorems \ref{thm:err.est} and \ref{L2:err.est}, and provide a few numerical examples.
In Section \ref{sec:acreac} we prove the local equilibrium properties of the HHO method and identify discrete equilibrated counterparts of moments and shear forces at interfaces.
Section \ref{sec:analysis} collects the technical proofs of the properties of the discrete bilinear form relevant to the analysis. Conclusions and perspectives are discussed in Section \ref{sec:conclusioni}.


\section{Discrete setting}\label{sec:discset}

In this section we introduce some assumptions on the mesh, recall a few known results, and define two projectors on local polynomial spaces that will play a key role in the analysis of the method.

\subsection{Mesh}\label{sec:discset:mesh}
The HHO method is built upon a polygonal mesh of the domain $\Omega$ defined prescribing a \emph{set of elements} $\cc T_h$ and a \emph{set of faces} $\cc F_h$.

The set of elements $\cc T_h$ is a finite collection of open disjoint polygons $T$ with nonzero area such that $\overline\Omega = \bigcup_{T\in \cc T_h} \overline T$ and $h\coloneq \max_{T\in \cc T_h} h_T$, with $h_T$ denoting the diameter of $T$.
The set of faces $\cc F_h$ is a finite collection of open disjoint line segments in $\overline{\Omega}$ with nonzero length such that, for all $F\in\cc F_h$, (i) either there exist two distinct mesh elements $T_1,T_2\in\cc T_h$ such that $F\subset\partial T_1\cap\partial T_2$ (and $F$ is called an {\em interface}) or (ii) there exists a mesh element $T\in\cc T_h$ such that $F\subset\partial T\cap\partial\Omega$ (and $F$ is called a {\em boundary face}).
We assume that $\cc F_h$ is a partition of the mesh skeleton in the sense that $\bigcup_{T\in\cc T_h}\de T=\bigcup_{F\in\cc F_h} \overline{F}$.

We denote by $\cc F_h^{\rm i}$ the set of all interfaces and by $\cc F_h^{\rm b}$ the set of all boundary faces, so that $\cc F_h = \cc F_h^{\rm i} \cup \cc F_h^{\rm b}$. The length of a face $F\in \cc F_h$ is denoted by $h_F$. For any $T \in \cc T_h$, $\cc F_T$ is the set of faces that lie on $\partial T$ (the boundary of $T$) and, for any $F\in \cc F_T$, $\bm n_{TF}$ is the unit normal to $F$ pointing out of $T$. Symmetrically, for any $F \in \cc F_h$, $\cc T_F$  is the set containing the mesh elements sharing the face $F$ (two if $F$ is an interface, one if $F$ is a boundary face).

The notion of {\em geometric regularity} for polygonal meshes is more subtle than for standard meshes.
To formulate it, we assume the existence of a {\em matching simplicial submesh}, meaning that there is a conforming triangulation $\mathfrak{T}_h$ of the domain such that each mesh element $T\in\cc T_h$ is decomposed into a finite number of triangles from $\mathfrak{T}_h$ and each mesh face $F\in\cc F_h$ is decomposed into a finite number of edges from the skeleton of $\mathfrak{T}_h$.
We denote by $\varrho$ the {\em regularity parameter} such that (i) for any triangle $S\in\mathfrak{T}_h$ of diameter $h_S$ and inradius $r_S$, $\varrho h_S\le r_S$ and (ii) for any mesh element $T\in\cc T_h$ and any triangle $S\in\mathfrak{T}_h$ such that $S\subset T$, $\varrho h_T \le h_S$.
When considering refined mesh sequences, the regularity parameter should remain bounded away from zero.

In what follows, we also assume that the mesh is {\em compliant with the data}, i.e., for each mesh element $T\in\cc T_h$ there exist a \emph{unique} polygon $\Omega_i\in P_\Omega$ {(see \eqref{eq:POmega})} such that $T\subset\Omega_i$.
As a result, the material tensor field $\bb A$ is \emph{element-wise constant}, and we set for the sake of brevity
$$\bb A_T \coloneq \bb A_{| T}\qquad\forall T\in \cc T_h.$$
We also denote by $\cc A_T^-$ and $\cc A_T^+$, respectively, the smallest and largest eigenvalues of $\bb A_T$, regarded as an endomorphism
of $\bb R^{2\times2}_{\mr{sym}}$. For $l\ge 0$ we also introduce, for later use, the \emph{broken Sobolev space}
\begin{equation}\label{H_ell_brise}
H^l(\cc T_h)\coloneq\left\{v\in L^2(\Omega)\,:\, v_{|T}\in H^l(T)\quad\forall T\in\cc T_h\right\},
\end{equation}
equipped, unless noted otherwise, with the broken norm $\|{\cdot}\|_{H^l(\cc T_h)}$ defined by
\begin{equation}\label{norme_brisee}
\forall v\in H^{l}(\cc T_h), \quad \|v\|_{H^{l}(\cc T_h)}\coloneq\left(\sum_{T\in\cc T_h}\|v\|_{H^{l}(T)}^2\right)^{\nicefrac12} \!\! .
\end{equation}

\subsection{Basic results}
We next recall a few geometric and functional inequalities, whose proofs are straightforward adaptations of the results collected in \cite[Chapter 1]{Di-Pietro.Ern:12} (where a slightly different notion of mesh faces is considered).
For any mesh element $T\in\cc T_h$ and any face $F\in\cc F_T$ it holds that
\begin{equation}\label{meshreg}
\varrho^2 h_T \le h_F \le h_T,
\end{equation}
which expresses the fact that we are working on isotropic meshes.
Moreover, the maximum number of faces of a mesh element is uniformly bounded:
There is an integer $N_\partial\ge 3$ only depending on $\varrho$ such that
\begin{equation}\label{eq:bnd.faces}
  \max_{h\in\cc H}\max_{T\in\cc T_h}\card(\cc F_T)\le N_\partial.
\end{equation}
Let a polynomial degree $l\ge 0$ be fixed, let $X$ be a mesh element or face, and denote by $\bb P^l(X)$ the space spanned by the restrictions to $X$ of two-variate polynomials of total degree at most $l\ge 0$. 
There exist three real numbers $C_{\rm tr} > 0$, $C_{\rm tr, c} > 0$, and $C_{\rm inv}$ depending on $\varrho$ and possibly on $l$, but independent of $h$, such that
for any $T\in \cc T_h$ and $F\in \cc F_T$, the following discrete trace, continuous trace, and inverse inequalities hold:
\begin{subequations}
\begin{alignat}{2}
\|w\|_F & \le C_{\rm tr}\,h_F^{-\nicefrac{1}{2}}\,\|w\|_T  &\qquad& \forall w \in \bb P^l(T), \label{tr_discr}\\
h_T^{\nicefrac{1}{2}}\|w\|_{\de T} & \le C_{\rm tr, c}\left(\| w\|_T +h_T\|\grad w\|_T\right) &\qquad& \forall w \in H^1(T),\label{tr_cont}\\
\|\grad w\|_T & \le C_{\rm inv}h_T^{-1}\|w\|_T  &\qquad& \forall w \in \bb P^l(T) \label{inv},
\end{alignat}
\end{subequations}
We also recall the following Poincar\'{e} inequality, valid for all $T\in \cc T_h$ and all $ w \in H^1(T) $ such that $ (w,1)_T=0 $:
\begin{equation}\label{poincare}
\|w\|_{T} \le C_{\rm p} h_T \|\grad w\|_T,
\end{equation}
where the real number $C_{\rm p}$ is independent of both $h_T$ and $T$, but possibly depends on $\varrho$ (for instance, $C_{\rm p} = \pi^{-1}$ for convex elements \cite{bebendorf}).

\subsection{Projectors on local polynomial spaces}

Projectors on local polynomial spaces are an essential ingredient in the construction and analysis of our method.
Let a polynomial degree $l\ge 0$ be fixed, and let $X$ denote a mesh element or face.
The $L^2$-orthogonal projector $\pi_X^l\colon L^2(X) \to \bb P^l (X)$ is such that, for all $v\in L^2(X)$, $\pi_X^l v$ is the unique polynomial satisfying the relation
\begin{equation}\label{L2proj}
(\pi_X^l v-v,w)_X = 0\qquad \forall w \in \bb P^l (X).
\end{equation}
The corresponding vector-valued version, denoted by $\bm{\pi}_X^l$, acts component-wise.
We recall the following approximation results that are a special case of the ones proved in \cite[Lemmas 3.4 and 3.6]{Di-Pietro.Droniou:17*1}:
There exists a real number $C > 0$ independent of $h$, but possibly depending on $\varrho$ and $l$, such that, for all $T \in \cc T_h$, all $s \in\{0,\dots,l+1\}$, and all $v \in H^s(T)$,
\begin{subequations}
\begin{equation}
  | v - \pi_T^l v|_{H^m(T)}  \le C h_T^{s-m} |v|_{H^s(T)}\qquad \forall m \in \{0,\dots,s\}, \label{stima}
\end{equation}
and, if $s\ge 1$,
\begin{equation}
 | v - \pi_T^l v |_{H^m(\de T)} \le C h_T^{s-m-\nicefrac12} |v|_{H^s(T)}\qquad \forall m \in \{0,\dots,s-1\}. \label{stima_bordo}
\end{equation}
\end{subequations}
Here we have set, for any $\pphi \in H^s(T)$, 
$$
|\pphi|_{H^m(T)} \coloneq \sum_{\bm \alpha\in\bb N^2,\|\bm \alpha\|_1 = m} \|\de^{\bm\alpha} \pphi\|_{L^2(T)},
\qquad
|\pphi|_{H^m(\de T)} \coloneq \sum_{\bm \alpha\in\bb N^2,\|\bm \alpha\|_1 = m} \|\de^{\bm\alpha} \pphi\|_{L^2(\de T)},
$$
with $m$ respectively as in \eqref{stima} and \eqref{stima_bordo}, $\|\bm\alpha\|_1 \coloneq \alpha_1 + \alpha_2$ and $\de^{\bm \alpha} \coloneq \de_1^{\alpha_1}\de_2^{\alpha_2}$.
Notice that, in the second definition, $\pphi$ and $\de^{\bm\alpha} \pphi$ stand for the \emph{boundary traces} of the function and of its
derivatives up to order $m$, respectively.

Let a mesh element $T\in\cc T_h$ be fixed.
For $u,v\in H^{2}(T)$, we let $a_{|T}(u,v) \coloneq \left(\bb A_T \grad^2 u,\grad^2 v\right)_T$ and introduce the \emph{{local energy projector}} $\varpi_T^l:H^2(T)\to\bb P^l(T)$ such that, for any integer $l\ge 2$ and any function $v\in H^2(T)$,
\begin{equation}\label{eq:biell}
  \text{$a_{|T}(\varpi_T^l v-v, w)_T = 0$ for all $w\in\bb P^l(T)$\quad~and\quad$\pi_T^1 (\varpi_T^l v-v)= 0$.}
\end{equation}
Optimal approximation properties for the {local energy projector} are stated in the following theorem, whose proof is given in Section \ref{sec:approx.biell}.

\begin{theorem}[Optimal approximation properties of the {local energy projector}]\label{thm:approx.biell}
  There is a real number $C>0$ independent of $h$, but possibly depending on $\bb A$, $\varrho$ and $l$, such that, for all $T\in\cc T_h$, all $s\in\{2,\ldots,l+1\}$,
  and all $v \in H^{s}(T)$, it holds
  \begin{subequations}\label{approx_biell}
    \begin{equation}
      |v-\varpi_T^l v |_{H^m(T)} \le C h_T^{s-m} |v|_{H^{s}(T)} \qquad \forall m \in \{0,\dots,s\},
      \label{approx_properties}
    \end{equation}
    and
    \begin{equation}
      |v-\varpi_T^l v|_{H^m(\de T)} \le C h_T^{s-m-\nicefrac12} |v|_{H^{s}(T)} \qquad \forall m \in \{0,\dots,s-1\}.
      \label{approx_properties_b}
    \end{equation}
  \end{subequations}
\end{theorem}
\begin{remark}[Dependence on the material tensor]
It can be checked that the constant $C$ in the right-hand side of \eqref{approx_biell} actually depends
on $\bb A_T$ only through the square root of the ratio between $\cc A_T^+$ and $\cc A_T^-$.
\end{remark}

\section{The Hybrid High-Order method}\label{sec:hho}

In this section, we present the construction underlying the HHO method, state the discrete problem, and discuss the main results. {Henceforth, we fix \emph{once and for all} a polynomial degree $k \ge 1$.}

\subsection{Local discrete unknowns and interpolation}

Let a mesh element $T\in\cc T_h$ be fixed. The \emph{local space of discrete unknowns} is defined as the set
\begin{equation}\label{local_ddf}
\ul{\mr U}_T^k \coloneq \bb P^{k} (T) \times \left(
\bigtimes_{F\in \cc F_T}  \bb P^{k}(F)^2\right)
\times \left(
\bigtimes_{F \in \cc F_T} \bb P^{{k}} (F)
\right).
\end{equation}
For a general collection of discrete unknowns $\ul{\mr v}_T \in \ul{\mr U}_T^k$, we use the standard underlined HHO notation
$$
\ul{\mr v}_T = (v_T, \left(\bm v_{\grad,F}\right)_{F\in\mathcal{F}_T},(v_F)_{F\in\mathcal{F}_T}),
$$
where $v_T$ contains the element-based discrete unknowns, $\bm v_{\grad,F}$ the discrete unknowns related to the trace of the gradient on the face $F$, and $v_F$ the discrete unknowns related to the trace on $F$.

The \emph{local interpolation operator} $\ul{\mr I}_T^k \colon H^2(T) \to \ul{\mr U}_T^k$ is such that, for all $v \in H^2(T)$,
\begin{equation}\label{interpolation_locale}
  \ul{\mr I}_T^k v\coloneq\left(
  \pi^k_T  v, (\bm{\pi}_F^{k}(\grad v)_{| F})_{F\in \cc F_T}, (\pi_F^{{k}} (v_{| F} ))_{F\in \cc F_T}
  \right).
\end{equation}
Since the boundary of $T$ is piecewise smooth (see, e.g., \cite{necas}), the trace theorem ensures that the restrictions $v_{| F}$ and $(\grad v)_{| F}$ of $v$ appearing in~\eqref{interpolation_locale} are both well-defined.


\subsection{Local deflection reconstruction}\label{sec:locrec}

Let again a mesh element $T\in\cc T_h$ and a polynomial degree $k \ge 1$ be fixed.
We introduce the \emph{local deflection reconstruction operator} $\pT \colon \ul{\mr U}_T^k \to \bb P^{k+2}(T)$ such that, for all $\ul{\mr v}_T\in  \ul{\mr U}_T^k$, $\pT \ul{\mr v}_T \in \bb P^{k+2}(T)$ satisfies for all $w \in \bb P^{k+2}(T)$
\begin{multline}\label{reconstruction_locale}
  a_{|T}(\pT \ul{\mr v}_T,w)= \\
  -(v_T,\div\divb\bm M_{w,T})_T
  -\sum_{F\in\cc F_T} \left(\bm v_{\grad,F},\bm M_{w,T}\bm n_{TF}\right)_F
  +\sum_{F\in \cc F_T} \big(v_F, \divb \bm M_{w,T} \cdot \bm n_{TF}\big)_F,
\end{multline}
where $\bm M_{w,T} \coloneq -\bb A_T\grad^2 w$.
Here, the notation $\bm M_{w,T}$ is used to emphasize the fact that $\bm M_{w,T}$ is a moment tensor of \emph{virtual}
nature (with space of virtual deflections equal to $\bb P^{k+2}(T)$) unlike tensor $\bm M$
appearing in bilinear form $a$ introduced in \eqref{varform_hho}.
The right-hand side of~\eqref{reconstruction_locale} is conceived so as to resemble an integration by parts formula where the roles of the function represented by $\ul{\mr v}_T$ and of its gradient are played by element discrete unknowns inside volumetric integrals and by face-based discrete unknowns on boundary integrals.

Since ${\ker}\grad^2=\bb P^1(T)$, the compatibility condition for problem~\eqref{reconstruction_locale} requires that the linear form on the right-hand side vanish on the elements of $\bb P^1(T)$; since $\bm M_{w,T} = \bm 0$ for all $w \in \bb P^1(T)$, this condition is satisfied.
The solution of \eqref{reconstruction_locale} is not unique: if $\pT \ul{\mr v}_T \in \bb P^{k+2}(T)$ is
a solution, $\pT \ul{\mr v}_T+z_T$ for any $z_T \in \bb P^1(T)$ also is.
To ensure uniqueness, we add the \emph{closure condition}
\begin{equation}\label{fermeture}
\pi^1_T \pT \ul{\mr v}_T = \pi^1_T v_T.
\end{equation}
Notice, in passing, that element discrete unknowns do not contribute to the right-hand side of~\eqref{reconstruction_locale} for $k=1$, and they only appear in the closure condition~\eqref{fermeture}.

For further use, we also observe that, since $v_T$ is smooth, performing an integration by parts on the first term in the right-hand side of \eqref{reconstruction_locale} and using the symmetry of $\bb A_T$ leads to the following reformulation, which points out the non-conformity of the method:
\begin{multline}\label{reconstruction_locale2}
  a_{|T}(\pT \ul{\mr v}_T,w)=\\
  a_{|T}( v_T,w)
  - \sum_{F\in\cc F_T} \left(\bm v_{\grad,F}-\grad v_T,\bm M_{w,T}\bm n_{TF}\right)_F
  + \sum_{F\in \cc F_T} \big(v_F-v_T, \divb \bm M_{w,T} \cdot \bm n_{TF}\big)_F.
\end{multline}

The definition of $\pT$ is justified by the following proposition, which establishes a link with the {local energy projector} defined
by~\eqref{eq:biell}.
\begin{proposition}[Link with the {local energy projector}]\label{polynomial_consistency}
  It holds
  \begin{equation}\label{eq:pTIT=varpi}
    \pT\circ\ul{\mr I}_T^k = \varpi_T^{k+2}.
  \end{equation}
\end{proposition}
\begin{proof}
  We write~\eqref{reconstruction_locale} for $\ul{\mr v}_T=\ul{\mr I}_T^k v$ (cf.~\eqref{interpolation_locale} for the definition of the local interpolator).
  Since $w \in \bb P^{k+2}(T)$ and $\bb A_T$ is a constant tensor, we infer that
  $$
  \div\divb \bm M_{w,T} \in \bb P^{k-2}(T)\subset \bb P^{k}(T)
  $$
  and, for all $F\in\cc F_T$,
  $$
  (\bm M_{w,T})_{|F} \bm n_{TF}\in \bb P^k(F)^2,\quad
  (\divb \bm M_{w,T})_{|F} \cdot \bm n_{TF}  \in \bb P^{k-1}(F) \subset \bb P^{k}(F).
  $$
  Consequently, recalling the definition~\eqref{L2proj} of $\pi^k_T  $, $\bm{\pi}_F^k$, and $\pi_F^k$, we have
  $$
  \begin{alignedat}{2}
    (\pi^k_T   v,\div\divb\bm M_{w,T})_T & = (v, \div\divb\bm M_{w,T})_T,
    \\
    (\bm{\pi}_F^{k}(\grad v)_{| F},\bm M_{w,T}\bm n_{TF})_F & = ((\grad v)_{| F},\bm M_{w,T}\bm n_{TF})_F,
    \\
    (\pi_F^{k} v_{|F}, \divb \bm M_{w,T} \cdot \bm n_{TF})_F & = (v _{| F}, \divb \bm M_{w,T} \cdot \bm n_{TF})_F .
  \end{alignedat}
  $$
  Plugging the above identities into the right-hand side of~\eqref{reconstruction_locale}, performing an integration by parts, and using the symmetry of $\bb A_T$, we arrive at the following \emph{orthogonality condition}:
  \begin{equation}\label{orthogonality}
    a_{|T}(\pT \ul{\mr I}_T^k v-v,w)  = 0.
  \end{equation}
  Comparing~\eqref{orthogonality} and \eqref{fermeture} with the definition~\eqref{eq:biell} of $\varpi_T^{k+2}$ concludes the proof.
\end{proof}

\begin{remark}[Approximation properties for $\pT\circ\ul{\mr I}_T^k$]
\label{pT_circ_IT}
The above result implies that $\pT\circ\ul{\mr I}_T^k=\varpi_T^{k+2}$ has optimal approximation properties in $\bb P^{k+2}(T)$, in the sense made precise by Theorem~\ref{thm:approx.biell}.
\end{remark}


\subsection{Local contribution}\label{sec:loccontr}
We introduce the local discrete bilinear form $\mr{a}_{T}(\cdot,\cdot)$ on $\ul{\mr U}_T^k \times \ul{\mr U}_T^k$ given by
\begin{equation}\label{forma_locale}
\mr{a}_{T} (\ul{\mr u}_T, \ul{\mr v}_T) \coloneq
a_{| T}( \pT \ul{\mr u}_T,\pT \ul{\mr v}_T) + \mr{s}_T(\ul{\mr u}_T, \ul{\mr v}_T).
\end{equation}
Here, the first contribution is the usual Galerkin term responsible for consistency.
The second contribution, in charge of stability, penalizes high-order differences between the reconstruction and the unknowns and is such that, for all $ (\ul{\mr u}_T,\ul{\mr v}_T) \in \ul{\mr U}_T^k \times \ul{\mr U}_T^k $,
\begin{equation}\label{stabiliz}
  \begin{aligned}
\mr{s}_T(\ul{\mr u}_T,\ul{\mr v}_T)&\coloneq
\frac{\cc A^+_T}{h_T^4}\Big(\pi^k_T  (\pT\ul{\mr u}_T - u_T),\pi^k_T  (\pT\ul{\mr v}_T - v_T) \Big)_T \\
&\quad + \frac{\cc A^+_T}{h_T} \sum_{F\in \cc F_T} \Big(\bm\pi^k_F(\grad \pT\ul{\mr u}_T - \bm u_{\grad,F}),
\bm\pi^k_F(\grad \pT\ul{\mr v}_T - \bm v_{\grad,F}) \Big)_F \\
&\quad + \frac{\cc A^+_T}{ h_T^3} \sum_{F\in \cc F_T} \Big( \pi^k_F(\pT\ul{\mr u}_T - u_F), \pi^k_F(\pT\ul{\mr v}_T - v_F)\Big) _F.
  \end{aligned}
\end{equation}
\begin{remark}[Stabilization]
  Other expressions are possible for the stabilization term, and the specific choice can affect the accuracy of the results.
  In particular, the discussion below remains true if we replace \eqref{forma_locale} by
  \begin{equation}\label{forma_locale.eta}
    \mr{a}_{T} (\ul{\mr u}_T, \ul{\mr v}_T) \coloneq
    a_{| T}( \pT \ul{\mr u}_T,\pT \ul{\mr v}_T) + \eta\mr{s}_T(\ul{\mr u}_T, \ul{\mr v}_T),
  \end{equation}
  with $\eta>0$ denoting a user-dependent parameter independent of $h$.
  In practice, it is important that the numerical results be only marginally affected by the specific choice of the stabilization.
  We refer the reader to Section \ref{sec:numerical.examples} below for a numerical study of the robustness of the method with respect to $\eta$.
\end{remark}
The following proposition states a consistency result for the stabilization bilinear form \eqref{stabiliz}.
\begin{proposition}[Consistency of $\mr s_T$] There is a real number $C>0$ independent of $h$, but possibly depending on $\bb A$, $\varrho$ and $k$, such that, for all $v\in H^{k+3}(T)$,
\begin{equation}\label{consistenza_s_T}
\mr s_T(\ul{\mr I}^k_T v, \ul{\mr I}^k_T v)^{\nf{1}{2}} \le C h_T^{k+1} |v|_{H^{k+3}(T)}.
\end{equation}
\end{proposition}
\begin{proof}
We have $$\mr s_T(\ulr I^k_T v, \ulr I^k_T v) = \mathfrak T_1 + \mathfrak T_2 + \mathfrak T_3,$$ where, recalling Proposition~\ref{polynomial_consistency} and using the linearity and idempotency of projectors,
$$
\begin{aligned}
\mathfrak T_1 & \coloneq \frac{\cc A^+_T}{h_T^{4}} \| \pi^k_T   ( \varpi^{k+2}_T v - \pi^k_T   v)\|_T^2 = 
\frac{\cc A^+_T}{h_T^{4}} \|\pi^k_T   (\varpi^{k+2}_T v - v)\|_T^2,\\
\mathfrak T_2 & \coloneq \frac{\cc A^+_T}{h_T} \sum_{F\in \cc F_T} \| \bm\pi^k_F(\grad \varpi^{k+2}_T v - \bm\pi^k_F(\grad v))\|_F^2 = 
\frac{\cc A^+_T}{h_T} \sum_{F\in\cc F_T} \|\bm\pi^k_F(\grad \varpi^{k+2}_T v - \grad v) \|_F^2 , \\
\mathfrak T_3 & \coloneq \frac{\cc A^+_T}{ h_T^3} \sum_{F\in \cc F_T} \|\pi^k_F(\varpi^{k+2}_T v - \pi^k_F v)\|_F^2 = 
\frac{\cc A^+_T}{ h_T^3} \sum_{F\in \cc F_T} \|\pi^k_F(\varpi^{k+2}_T v - v)\|_F^2.
\end{aligned}
$$
By the boundedness of $L^2$-projectors, along with the approximation properties \eqref{approx_properties}--\eqref{approx_properties_b} of
$\varpi^{k+2}_T$ with $s=k+3$ and, respectively, $m=0$ for $\mathfrak T_1$, $m=1$ for $\mathfrak T_2$, and again $m=0$ for $\mathfrak T_3$, the conclusion follows.
\end{proof}
We equip the space $\ul{\mr U}^k_T$ with the following local discrete seminorm: 
\begin{equation}\label{localseminorm}
\|\ul{\mr v}_T\|_{\bb A,T}^2 \coloneq \|\bb A_T^{\nicefrac{1}{2}} \grad^2 v_T\|_T^2\, + \frac{\cc A^+_T}{h_T} \sum_{F \in \cc F_T} \|\bm v_{\grad,F}-\grad v_T\|_F^2\, +
\frac{\cc A^+_T}{h_T^3} \sum_{F\in \cc F_T}
 \|v_F - v_T\|_F^2.
\end{equation}
The following result shows that the bilinear form $\mr{a}_T$ induces on $\ul{\mr U}_T^k$ a seminorm $\|{\cdot}\|_{\mr{a},T}$ uniformly equivalent to $\|{\cdot}\|_{\bb A,T}$.
\begin{lemma}[Local coercivity and boundedness]\label{coerciv}
  There is a real number $C>0$ independent of $h$, but possibly depending on $\bb A$, $\varrho$ and $k$, such that, for all $T\in\cc T_h$, the following inequalities hold (expressing, respectively, the coercivity and boundedness of $\mr{a}_T$):
  \begin{equation}\label{equivnorm}
    C^{-1} \| \ul{\mr v}_T \|_{\bb A,T}^2 \le
    \|\ul{\mr v}_T\|^2_{\mr{a},T}\coloneq\mr{a}_{T}(\ul{\mr v}_T, \ul{\mr v}_T)
    \le C\| \ul{\mr v}_T \|_{\bb A,T}^2
    \qquad\forall\ul{\mr v}_T \in \ul{\mr U}_T^k.
  \end{equation}
\end{lemma}
\begin{proof}
  See Section~\ref{sec:analysis:well-posedness}.
\end{proof}
{
\begin{remark}[Polynomial degree]
\label{remark:polynomial_degree}
The assumption $k \ge 1$ is \emph{essential} in the proof of the above result. For this reason, the steps in which this hypothesis is used are pointed out accordingly.
\end{remark}
}
%

\subsection{Global space, interpolation, and norm}
We define the following \emph{global space of discrete unknowns}:
\begin{equation}\label{global_ddf}
\ul{\mr U}_h^k \coloneq \left(\bigtimes_{T\in \cc T_h} \bb P^{k} (T) \right)
\times \left(
\bigtimes_{F\in \cc F_h} \bb P^{k}(F)^2\right)
\times \left(
\bigtimes_{F \in \cc F_h} \bb P^k (F)
\right).
\end{equation}
Note that interface unknowns in $\ul{\mr U}_h^k$ are single-valued, i.e., their values match from one element to the adjacent one.
For a collection of discrete unknowns in $\ul{\mr U}_h^k$, we use the notation
$$
\ul{\mr v}_h = \left((v_T)_{T\in \cc T_h}, (\bm v_{\grad,F})_{F\in \cc F_h}, (v_F)_{F\in \cc F_h} \right),
$$ and we denote by $\ul{\mr v}_T = \left(v_T, (\bm v_{\grad,F})_{F\in\cc F_T}, (v_F)_{F\in\cc F_T}\right)\in \ul{\mr U}^k_T$ its restriction to a mesh element $T\in \cc T_h$.
We also denote by $v_h$ (no underline) the broken polynomial function on $\cc T_h$ such that $$v_{h|T}=v_T\qquad\forall T\in\cc T_h.$$
We define the global interpolator $\ul{\mr I}_h^k \colon H^2(\Omega) \to \ul{\mr U}_h^k$ such that, for all $v\in H^2(\Omega)$,
\begin{equation}\label{global_interpolation}
(\ul{\mr I}_h^k v )_{| T} = \ul{\mr I}^k_T (v_{| T} )\quad \forall T \in \cc T_h.
\end{equation}
The space $\ul{\mr U}_h^k$ is equipped with the following seminorm (cf.~\eqref{localseminorm} for the definition of $\|{\cdot}\|_{\bb A,T}$):
\begin{equation}\label{norm discrete}
  \| \ul{\mr v}_h\|_{\bb A,h}^{2}\coloneq \sum_{T\in \cc T_h} \|\ul{\mr v}_T\|_{\bb A,T}^2.
\end{equation}

We notice that the couple of boundary conditions~\eqref{eq:bc1}--\eqref{eq:bc2} is \emph{equivalent} to the couple $u=0$ on $\de\Omega$ and $\grad u = \bm 0$ on $\de\Omega$.
Indeed, the fact that $u$ vanishes on $\de\Omega$ implies its tangential derivative to vanish on $\de\Omega$ as well.
Accounting for this remark, we introduce the following subspace that incorporates the latter couple of boundary conditions in a strong manner:
\begin{equation}\label{boundary_ddf}
\ul{\mr U}_{h,0}^k \coloneq \{ \ul{\mr v}_h\in \ul{\mr U}_h^k : \\
v_F = 0,\,\,\bm v_{\grad,F} = \bm 0 \hbox{ for any }F \in \cc F_h^{\rm b}
\}.
\end{equation}
It is a simple matter to check that the image of the restriction of $\ul{\mr I}_h^k$ to $H^2_0(\Omega)$ is contained in $\ul{\mr U}_{h,0}^k$.
\begin{proposition}[Norm $\|\ul{\mr v}_h\|_{\bb A,h}$]
  The mapping $\ul{\mr U}_{h,0}^k \ni \ul{\mr v}_h \mapsto \| \ul{\mr v}_h \|_{\bb A,h} \in \bb R$ defines a norm on $\ul{\mr U}_{h,0}^k$.
\end{proposition}
\begin{proof}
  The seminorm property is trivial.
  It then suffices to show that $ \| \ul{\mr v}_h \|_{\bb A,h} = 0 \Longrightarrow \ul{\mr v}_h = \ul 0 \in \ul{\mr U}^k_{h,0}$.
  Clearly, $ \| \ul{\mr v}_h \|_{\bb A,h} = 0 $ implies
  $\grad^2 v_T \equiv \bm 0$ for all $T \in \cc T_h$ and $\bm v_{\grad,F}-\grad v_T \equiv \bm 0$ and $v_F-v_T \equiv 0$  for all $F \in \cc F_h$.
    By definition \eqref{boundary_ddf}, we have $\bm v_{\grad,F} \equiv \bm 0$ and $v_F = 0$ for all $F \in \cc F^b_h$; thus, for any $T \in \cc T_h$,
    if $\cc F_T \cap \cc F^{\mr{b}}_h \neq \emptyset$ then there exists $F \in \cc F^{\mr{b}}_h$ such that $\grad v_T \equiv \bm 0$ and $v_T\equiv 0$ on $F$. Since $\grad^2 v_T \equiv \bm 0$ in $T$, these facts imply that $v_T \equiv 0$ in $T$, which in turn implies that $v_F \equiv 0$ and $\bm v_{\grad,F} \equiv \bm 0$ for 
    all $F \in \cc F_T$. Repeating this argument for inner layers of elements yields the assertion.
\end{proof}

\subsection{Discrete problem}

The \emph{discrete problem} is formulated as follows:
Find $\ul{\mr u}_h \in \ul{\mr U}^k_{h,0}$ such that
\begin{equation}\label{pb_discreto}
  \mr{a}_h(\ul{\mr u}_h, \ul{\mr v}_h) = (f, v_h)\qquad\forall \ul{\mr v}_h \in \ul{\mr U}^k_{h,0}
\end{equation}
with global bilinear form $\mr{a}_h$ on $\ul{\mr U}_h^k \times \ul{\mr U}_h^k$ obtained by element-by-element assembly setting
\begin{equation}
  \label{global_form}\mr{a}_h(\ul{\mr u}_h, \ul{\mr v}_h) \coloneq \sum_{T\in \cc T_h} \mr{a}_{T}(\ul{\mr u}_T, \ul{\mr v}_T).
\end{equation}
The following lemma summarizes the properties of the global bilinear form $\mr{a}_h$.
\begin{lemma}[Properties of $\mr{a}_h$]\label{stab_consist}
  The bilinear form $\mr{a}_h$ defined by~\eqref{global_form} has the following properties:
  \begin{enumerate}[(i)]
  \item \emph{Coercivity and boundedness.} There is a real number $C>0$ independent of $h$, but possibly depending on $\bb A$, $\varrho$ and $k$, such that
    \begin{equation}\label{global_stab}
      C^{-1}\| \ul{\mr v}_h\|_{\bb A,h}^2
      \le \|\ul{\mr v}_h\|_{\mr{a},h}^2\coloneq\mr{a}_h(\ul{\mr v}_h, \ul{\mr v}_h)
      \le C\| \ul{\mr v}_h\|_{\bb A,h}^2
      \qquad\forall \ul{\mr v}_h \in \ul{\mr U}_h^k.
    \end{equation}
  \item \emph{Consistency.} There is a real number $C>0$ independent of $h$, but possibly depending on $\bb A$, $\varrho$ and $k$, such that, for all \mbox{$v\in 
   H_0^2(\Omega) \cap H^4(\Omega) \cap H^{k+3}(\cc T_h)$}, it holds that
    \begin{equation}\label{consistenza}
      \sup_{\ul{\mr w}_h \in \ul{\mr U}_{h,0}^k\setminus\{\ul{0}\}} \frac{ ( \mr{div}\,\mr{\bf{div}}\, \bb A\grad^2 v, w_h) - 
        \mr{a}_h(\ul{\mr I}^k_h v,\ul{\mr w}_h)}{\|\ul{\mr w}_h\|_{\bb A,h}} \le C h^{k+1} |v|_{H^{k+3}(\Omega)}.
    \end{equation}
  \end{enumerate}
\end{lemma}
\begin{proof}
  See Section~\ref{sec:analysis:well-posedness}.
\end{proof}
As a consequence of the first inequality in~\eqref{global_stab}, the discrete problem~\eqref{pb_discreto} admits a unique solution. {%
This solution minimizes the following \emph{discrete energy}:
\begin{equation}\label{discr_en}
\ulr U^k_{h,0} \ni \ulr v_h \mapsto \mr E(\ulr v_h) \coloneq \frac{1}{2}\mr a_h(\ulr v_h,\ulr v_h) - (f,v_h) \in \bb R.
\end{equation}
The discrete energy will play an important role in numerical experiments (cf.~Section \ref{sec:numerical.examples} below).
}
{
\begin{remark}[Implementation]\label{rem:implementation}
Proceeding as in standard FE methods, to write an algebraic version of problem \eqref{pb_discreto} we associate to each mesh element or face a set of degrees of freedom (DOFs) that form a basis for the dual space of the local polynomial space supported by it.
  Let a basis $\mathcal{B}_h$ for the space $\ul{\mr U}_{h,0}^k$ be fixed such that every basis function is supported by only one mesh element or face.
  To fix the ideas, we take as DOFs the coefficients of the expansion of a HHO function $\ul{\mr v}_h\in\ul{\mr U}_{h,0}^k$ in $\mathcal{B}_h$, and we collect them in the vector $\sV$ partitioned as
  $$
  \sV=\begin{pmat}[{}]\sV[\cc T_h]\cr\- \sV[\cc F_h]\cr\end{pmat},
  $$
  where the subvector $\sV[\cc T_h]$ collects the coefficients associated with element-based DOFs, while the remaining coefficients (associated to face-based DOFs) are collected in $\sV[\cc F_h]$.
  Denote by $\sA$ the matrix representation of the bilinear form $\mathrm{a}_h$ and by $\sB$ the vector representation of the linear form $\ul{\mr v}_h\mapsto(f,v_h)$, both partitioned in a similar way.
The algebraic problem corresponding to \eqref{pb_discreto} reads
\begin{equation}\label{pb_algebrico}
  \underbrace{\begin{pmat}[{|}]
      \sA[\cc T_h\cc T_h] & \sA[\cc T_h\cc F_h] \cr\-
      \sA[\cc T_h\cc F_h]\trans & \sA[\cc F_h\cc F_h] \cr
  \end{pmat}}_{\sA}
  \underbrace{\begin{pmat}[{}]
      \sU[\cc T_h]\cr\-\sU[\cc F_h]\vphantom{\sU[\cc F_h]\trans}\cr
  \end{pmat}}_{\sU}=
  \underbrace{
  \begin{pmat}[{}]
      \sB[\cc T_h]\cr\-\mathsf{0}_{\cc F_h}\vphantom{\sU[\cc F_h]\trans}\cr
    \end{pmat}}_{\sB}.
\end{equation}
The submatrix $\sA[\cc T_h\cc T_h]$ is block-diagonal and symmetric positive definite, and is therefore inexpensive to invert.
In the practical implementation, this remark can be exploited by solving the linear system \eqref{pb_algebrico} in two steps:
  \begin{subequations}
  \begin{enumerate}[(i)]
  \item First, element-based coefficients in $\sU[\cc T_h]$ are expressed in terms of $\sB[\cc T_h]$ and $\sU[\cc F_h]$ by the inexpensive solution of the first block equation:
    \begin{equation}\label{eq:static.cond:1}
      \sU[\cc T_h]=\sA[\cc T_h\cc T_h]^{-1}\left(
      \sB[\cc T_h] - \sA[\cc T_h\cc F_h]\sU[\cc F_h]
      \right).
    \end{equation}      
    This step is referred to as {\em static condensation} in the FE literature;
  \item Second, face-based coefficients in $\sU[\cc F_h]$ are obtained solving the following global problem involving quantities attached to the mesh skeleton:
    \begin{equation}\label{eq:static.cond:2}
      \underbrace{\left(\sA[\cc F_h\cc F_h]-\sA[\cc T_h\cc F_h]\trans\sA[\cc T_h\cc T_h]^{-1}\sA[\cc T_h\cc F_h]\right)}_{\coloneq\tilde{\mathsf{A}}_{\cc F_h\cc F_h}}
      \sU[\cc F_h]
      =
      -\sA[\cc T_h\cc F_h]\trans\sA[\cc T_h\cc T_h]^{-1}\sB[\cc T_h].
    \end{equation}
    This computationally more intensive step requires to invert the symmetric matrix $\tilde{\mathsf{A}}_{\cc F_h\cc F_h}$, whose stencil involves neighbours through faces, and which has size $N_{\rm dof}\times N_{\rm dof}$ with $N_{\rm dof} = 2\card(\cc F_h^{\rm i}) {k + d - 1\choose k}$.
    Observing that $\tilde{\mathsf{A}}_{\cc F_h\cc F_h}$ is in fact the Schur complement of $\sA[\cc T_h\cc T_h]$ in $\sA$, and since $\sA$ is symmetric and both $\sA$ and $\sA[\cc T_h\cc T_h]$ are positive definite, a classical result in linear algebra yields that also $\tilde{\mathsf{A}}_{\cc F_h\cc F_h}$ is positive definite (see, e.g.,~\cite{horn_zhang}).
  \end{enumerate}
\end{subequations}
\end{remark}
}
\subsection{Main results}

We next present the main results of the analysis, namely error estimates in an energy-like norm, in a jump-seminorm, and in the $L^2$-norm.
Inside the proofs of this section, we often abridge as $a \lesssim b$ the inequality $a \le Cb$ with $C>0$ independent of $h$, but possibly depending on $\bb A$, $\varrho$, and $k$.

\subsubsection{Energy error estimate}

We introduce the global deflection reconstruction operator $\ph:\ul{\mr U}_h^k\to L^2(\Omega)$ such that, for all $\ul{\mr v}_h\in\ul{\mr U}_h^k$,
$$(\ph\ul{\mr v}_h)_{|T}=\pT\ul{\mr v}_T\qquad\forall T\in\cc T_h.
$$
We also define the stabilization seminorm $|{\cdot}|_{\mr{s},h}$ on $\ul{\mr U}_h^k$ setting, for all $\ul{\mr v}_h\in\ul{\mr U}_h^k$,
$$
|\ul{\mr v}_h|_{\mr{s},h}^2\coloneq\sum_{T\in\cc T_h}\mr{s}_T(\ul{\mr v}_T,\ul{\mr v}_T).
$$
\begin{theorem}[Energy error estimate]\label{thm:err.est}
  Let $u \in H^2_0(\Omega)$ and $\ul{\mr u}_h \in \ul{\mr U}^k_{h,0}$ denote the unique solutions to the continuous \eqref{varform_hho} and discrete \eqref{pb_discreto} problems, respectively.
  Assume the additional regularity $u \in H^4(\Omega) \cap H^{k+3}(\cc T_h)$.
  Then, it holds that
  \begin{equation}\label{errore'}
    \|\bb A^{\nicefrac12}\grad_h^2(\ph\ul{\mr u}_h - u)\| + |\ul{\mr u}_h|_{\mr{s},h}
    \le   C h^{k+1} |u|_{{H^{k+3}(\cc T_h)}},
  \end{equation}
  where $\grad_h$ denotes the usual broken gradient operator on $\cc T_h$ and the real number $C>0$ is independent of $h$ (but possibly depends on $\bb A$, $\varrho$, and $k$).
\end{theorem}
\medskip
\begin{remark}[Regularity of the solution]
Concerning the regularity assumptions on $u$, we mention as an example that, for $k=1$, the regularity $u\in H^4(\Omega)$ is satisfied by the solution of the biharmonic problem with Dirichlet boundary conditions (obtained taking $\bb A = \bb I$ in \eqref{static_hho}) posed on a three-dimensional cubic domain, provided the load $f$ is square-integrable (see, e.g., Maz'ya \cite[Chapter 4]{mazya}). In two dimensions, under the weaker assumption that $f \in H^{-1}(\Omega)$, it holds that $u \in H^3(\Omega)$ provided $\Omega$ is convex (see, e.g., Grisvard \cite[Chapter 3]{grisvard}). In general, a regularity assumption on the exact solution is actually the consequence of a \emph{compatibility condition} between the datum regularity and the domain geometry. When $f \in L^2(\Omega)$ in two dimensions,  one can have $u \in H^4(\Omega)$ under the condition of Kondratiev on the opening of each corner (see, e.g., Blum--Rannacher \cite{blum}, Grisvard \cite{grisvard2}). As a further reference on the regularity for the solution of fourth-order elliptic problems, we also refer the reader to Dauge \cite[Chapter 4]{dauge}.
{To close this remark, we emphasize that, since $u$ needs only be \emph{locally regular inside each element}, the presence of corner singularities and layers can be accounted for by a judicious choice of $h$ and, possibly, of $k$.}
\end{remark}
\begin{proof}[Proof of Theorem \ref{thm:err.est}]
  Let, for the sake of brevity, $\widehat{\ul{\mr u}}_h \coloneq \ul{\mr I}^k_h u$.
  We start by proving that
  \begin{equation}\label{errore}
    \|\ul{\mr u}_h - \widehat{\ul{\mr u}}_h\|_{\mr{a},h}
    \lesssim h^{k+1} |u|_{{H^{k+3}(\cc T_h)}},
  \end{equation}
  with norm $\|{\cdot}\|_{\mr{a},h}$ defined by~\eqref{global_stab}.
  Using the linearity of $\mr{a}_h$ in its first argument together with the discrete problem~\eqref{pb_discreto}, and recalling that $\div\divb \bb A \grad^2 u = f$ a.e. in $\Omega$, we have, for all $\ul{\mr v}_h \in \ul{\mr U}^k_{h,0}$,
  $$
  \mr a_h(\ul{\mr u}_h - \widehat{\ul{\mr u}}_h, \ul{\mr v}_h) = (f,v_h) - \mr a_h(\widehat{\ul{\mr u}}_h,\ul{\mr v}_h) \le \sup_{\ul{\mr w}_h \in \ul{\mr U}^k_{h,0}\setminus\{\ul{0}\}}
  \frac{(\div\divb\bb A\grad^2 u,w_h)-\mr a_h(\widehat{\ul{\mr u}}_h,\ul{\mr w}_h)}{\| \ul{\mr w}_h\|_{\bb A,h}} \|\ul{\mr v}_h\|_{\bb A,h}.
  $$
  Thus, choosing $\ul{\mr v}_h = \ul{\mr u}_h - \widehat{\ul{\mr u}}_h$ and using the consistency~\eqref{consistenza} of $\mr a_h$ to bound the supremum in the right-hand side, the basic estimate~\eqref{errore} follows.

  Let us now prove~\eqref{errore'}. Using the triangle inequality, we infer that
  \begin{equation*}
    \begin{aligned}
      &\|\bb A^{\nicefrac12}\grad_h^2(\ph\ul{\mr u}_h - u)\| + |\ul{\mr u}_h|_{\mr{s},h}
      \\
      &\qquad \le \|\bb A^{\nicefrac12}\grad_h^2(\ph\ul{\mr u}_h - \wh{\ul{\mr u}}_h)\| + |\ul{\mr u}_h - \widehat{\ul{\mr u}}_h|_{\mr{s},h}
      + \|\bb A^{\nicefrac12}\grad^2(\ph\wh{\ul{\mr u}}_h - u)\| + |\wh{\ul{\mr u}}_h|_{\mr{s},h}
      \\
      &\qquad \le \|\ul{\mr u}_h - \wh{\ul{\mr u}}_h\|_{\mr a,h}
      +\|\bb A^{\nicefrac12}\grad^2(\ph\wh{\ul{\mr u}}_h - u)\| + |\wh{\ul{\mr u}}_h|_{\mr{s},h},
    \end{aligned}
  \end{equation*}
  where we have used the discrete Cauchy--Schwarz inequality together with the definition~\eqref{global_stab} of the $\|{\cdot}\|_{\mr a,h}$-norm in the last line.
  The conclusion follows using~\eqref{errore} to estimate the first term in the right-hand side, the optimal approximation properties \eqref{approx_properties} of $\pT\wh{\ul{\mr u}}_T=\varpi_T^{k+2}u$ with $s = k+3$ and $m=2$ for all $T\in\cc T_h$ to estimate the second term,  and the consistency~\eqref{consistenza_s_T} of $\mr{s}_T$ for all $T\in\cc T_h$ to estimate the third term.
\end{proof}
{\begin{remark}[Convergence of face unknowns]
    Combining the norm equivalence \eqref{global_stab} with \eqref{errore}, we readily infer that
    \begin{multline*}
    \left[
      \sum_{T\in\cc T_h}\left( \|\bb A_T^{\nicefrac{1}{2}} \grad^2 (u_T - \pi_T^k u)\|_T^2
      + \sum_{F\in\cc F_T}\left(
      \frac{\cc A^+_T}{h_T}\|\bm u_{\grad,F}-\bm\pi_F^k(\grad u)\|_F^2
      + \frac{\cc A^+_T}{h_T^3}\|u_F - \pi_F^ku\|_F^2
      \right) \right)
      \right]^{\nicefrac12}
    \\
    \lesssim h^{k+1}|u|_{H^{k+3}(\cc T_h)},
    \end{multline*}
    which shows, in particular, that the face variables converge in an energy-like norm to the corresponding projections of the exact solution and its normal derivative.
    This is in itself a supercloseness result for the face variables, since, replacing $\bm\pi_F^k(\grad u)$ by $\grad u$ and $\pi_F^ku$ by $u$ in the left-hand side of the above inequality, one would only obtain a suboptimal estimate in $h^{k-1}$ (which would only converge if $k\ge 2$). 
    An optimal error estimate in $h^{k+1}$ for the trace of $u$ and its gradient can be recovered using the deflection reconstruction instead of the face variables:
    $$
    \left[
      \sum_{T\in\cc T_h}\sum_{F\in\cc F_T}\left(
      \frac{\cc A^+_T}{h_T}\|\grad(\pT\ul{\mr u}_T- u)\|_F^2
      + \frac{\cc A^+_T}{h_T^3}\|\pT\ul{\mr u}_T- u\|_F^2
      \right)
      \right]^{\nicefrac12}\lesssim h^{k+1}|u|_{H^{k+3}(\cc T_h)}.
    $$
\end{remark}}
\begin{remark}[Convergence of the jumps]
From the estimate of Theorem~\ref{thm:err.est}, one can prove that the jumps of $\ph\ul{\mr u}_h$ and of its gradient converge to zero with optimal rate. To this end, define on $H^2(\cc T_h)$ (cf. definition~\eqref{H_ell_brise}) the following \emph{jump seminorm}:
  $$
  |v|_{\mr{J},h}^2\coloneq
  \sum_{F\in\cc F_h}\left(
  \frac{\cc A_F}{h_F}\|\bm{\pi}_F^k\jump{\grad v}\|_F^2
  + \frac{\cc A_F}{h_F^3}\|\pi_F^{k}\jump{v}\|_F^2
  \right),
  $$
  where $\jump{{\cdot}}$ is the usual jump operator if $F$ is an interface (the sign is irrelevant), whereas $\jump{\varphi}\coloneq\varphi_{|F}$ if $F$ is a boundary face, and $\cc A_F\coloneq \min_{T\in \cc T_F} \cc A_T^+$. Then, observing that $|\ph\ul{\mr u}_h|_{\mr{J},h}^2\le 2\varrho^{-6}|\ul{\mr u}_h|^2_{\mr{s},h}$ as a consequence of the triangle inequality together with~\eqref{meshreg}, it is inferred from~\eqref{errore'} that
  \begin{equation}\label{eq:jump.est}
    |\ph\ul{\mr u}_h|_{\mr{J},h}\le C h^{k+1} |u|_{{H^{k+3}(\cc T_h)}}.
  \end{equation}
  with real number $C>0$ independent of $h$, but possibly depending on $\bb A$, $\varrho$, and $k$.
\end{remark}
\subsubsection{$\vect L^2$-error estimate}
A sharp $L^2$-norm error estimate can also be inferred assuming \emph{biharmonic regularity}, in the following form:
For all $q \in L^2(\Omega)$, the unique solution $z\in H^2_0(\Omega)$ to
\begin{equation}\label{aux_pb}
a(z,v) = (q,v)\qquad\forall v \in H^2_0(\Omega)
\end{equation}
satisfies the \emph{a priori} estimate {(see, e.g., \cite{blum})}
\begin{equation}\label{biell_reg}
\|z\|_{H^4(\Omega)} \le C_{\mr{bihar}} \|q\|,
\end{equation}
with $C_{\mr{bihar}}>0$ only depending on $\Omega$ and on $\bb A$.
\begin{theorem}[$L^2$-error estimate]\label{L2:err.est}
  Let $u \in H^2_0(\Omega)$ and $\ul{\mr u}_h \in \ul{\mr U}^k_{h,0}$ denote the unique solutions to the continuous \eqref{varform_hho} and discrete \eqref{pb_discreto} problems, respectively.
  Assume $u \in H^4(\Omega) \cap {H^{k+3}(\cc T_h)}$, biharmonic regularity, and $f \in H^{k+1}(\cc T_h)$.
  Then, there exists a real number $C>0$ depending on $\bb A$, $\varrho$, and $k$, but independent of $h$, such that
\begin{equation} \label{L2:err.est.thm}
  \|\ph \ulr u_h - u\| \le C h^{k+3} \left(\|u\|_{{H^{k+3}(\cc T_h)}}+\|f\|_{H^{k +1}(\cc T_h)}\right).
\end{equation}
\end{theorem}
\begin{proof}
Let, for the sake of brevity, $\wh{\ulr u}_h \coloneq \ulr I^k_h u$.
By the triangle inequality, we have that
$$
\|\ph \ulr u_h - u\| \le  \|\ph \wh{\ulr u}_h - u\| + \| \ph(\ulr u_h - \wh{\ulr u}_h)\|  \eqcolon \mathfrak T_1 + \mathfrak T_2.
$$
By the approximation properties \eqref{approx_properties}
of $\pT\circ \ulr I^k_T = \varpi^{k+2}_T$ (cf. Remark~\ref{pT_circ_IT}) with $s=k+3$ and $m=0$, we immediately
have that $$\mathfrak T_1 \lesssim h^{k+3} \|u\|_{{H^{k+3}(\cc T_h)}}.$$
For the second term, on the other hand, we observe that
$$
\begin{aligned}
\mathfrak T_2^2 & = \sum_{T\in\cc T_h} \|\pT(\wh{\ulr u}_T - \ulr u_T)\|_T^2 \\
& \lesssim \sum_{T\in\cc T_h}\left( h_T^4 \|\bb A^{\nf{1}{2}}_T \grad^2 \pT( \wh{\ulr u}_T - \ulr u_T)\|_T^2 +
\|\pi^1_T(\wh{u}_T - u_T)\|_T^2 \right) \\
& \lesssim h^4 \|\wh{\ulr u}_h - \ulr u_h\|_{\mr{a},h}^2 + \|\wh u_h - u_h\|^2,
\end{aligned}
$$
where we have used the triangle inequality and the approximation properties of $\pi_T^1$ for $s=2$ and $m=0$,
as well as the closure condition \eqref{fermeture} to pass to the second line, and the definition
of the $\|{\cdot}\|_{\mr a,h}$-norm to conclude. Using \eqref{errore} and Lemma~\ref{L2.lemma} below
to bound the first and second addend in the right-hand side, respectively, the conclusion follows.
\end{proof}
%
%

The following lemma, used in the proof of Theorem~\ref{L2:err.est} above, shows that element-based discrete unknowns behave ``almost'' like the $L^2$-orthogonal projection of the exact solution on the space of broken polynomials of total degree at most $k$ on $\cc T_h$.
\begin{lemma}[Supercloseness of element discrete unknowns]\label{L2.lemma}
  Under the assumptions and notations of Theorem \ref{L2:err.est}, it holds that
  \begin{equation} \label{L2:err.est.lemma}
    \|\wh u_h - u_h\| \le C h^{k+3} \left(\|u\|_{{H^{k+3}(\cc T_h)}}+\|f\|_{H^{k+1}(\cc T_h)}\right),
  \end{equation}
  where $\wh u_h$ and $u_h$ are the broken polynomial functions of total degree at most $k$ such that $\wh u_{h | T} \coloneq \wh u_T = \pi^k_T   u$
  and $u_{h | T} \coloneq u_T$ for any mesh element $T \in \cc T_h$.
\end{lemma}
\begin{proof}
  Set, for the sake of brevity, $\ulr e_h \coloneq \wh{\ulr u}_h - \ulr u_h$ and $e_h \coloneq \wh u_h - u_h$.
  Let $z$ solve \eqref{aux_pb} with $q = e_h$ and set $\wh{\ulr z}_h \coloneq \ulr I^k_h z$.
  Integrating by parts, using the linearity of $\mr a_h$ in its first argument, as well as the continuity
  of moments and shear forces at interfaces, and letting
  $\wc z_T \coloneq \varpi^{k+2}_T(z_{| T})$ for all $T\in\cc T_h$, we have that $\|e_h\|^2 = \mathfrak T_1 + \mathfrak T_2$, with
  \begin{equation}
    \begin{aligned}
      \mathfrak T_1 \coloneq & \sum_{T\in \cc T_h} \sum_{F\in \cc F_T}\Big( \big(\bb A_T \grad^2(z-\wc z_T)\big)\bm n_{TF},
      \bm e_{\grad,F}-\grad e_T)_F \\
      & -(\divb \bb A_T\grad^2(z-\wc z_T)\cdot \bm n_{TF}, e_F - e_T)_F\Big) - \sum_{T\in \cc T_h} \mr s_T(\wh{\ulr z}_T,\ulr e_T),\\
      \mathfrak T_2  \coloneq & \ \mr a_h(\wh{\ulr u}_h, \wh{\ulr z}_h) - (f,\pi^k_h z),
    \end{aligned}
  \end{equation}
  where $\pi^k_h$ is such that $(\pi^k_h v)_{| T} = \pi^k_T (v_{| T})$ for all $T\in \cc T_h$ and all $v \in H^2(\Omega)$.
  The Cauchy--Schwarz inequality then yields 
  $$
  \begin{aligned}
    |\mathfrak T_1| \lesssim
    \left[ \sum_{T \in \cc T_h}\hspace{-1ex} \left(
    \frac{h_T}{\cc A_T^+} \|\bb A_T \grad^2(z-\wc z_T)\|_{\partial T}^2
    + \frac{h_T^3}{\cc A_T^+} \|\divb \bb A_T\grad^2(z-\wc z_T)\|_{\partial T}^2
    \right) 
    + | \wh{\ulr z}_h|^2_{\mr s,h}
    \right]^{\nf{1}{2}}
    \hspace{-1em}\times\left( \|\ulr e_h\|_{\bb A,h}^2 + |\ulr e_h|_{\mr s,h}^2 \right)^{\nf{1}{2}}.
  \end{aligned}
$$
The approximation properties \eqref{approx_biell} of $\varpi^{k+2}_T$ with $s=4$, the consistency property \eqref{consistenza_s_T} of the stabilization bilinear form, and the stability of $\mr a_h$ together with the energy error estimate \eqref{errore} allow to conclude that
$$|\mathfrak T_1| \lesssim h^2 |z|_{H^4(\Omega)} h^{k+1} \|u\|_{{H^{k+3}(\cc T_h)}} \lesssim h^{k+3} \|u\|_{{H^{k+3}(\cc T_h)}} \|e_h\|,$$
where in the last estimate we have used the biharmonic regularity hypothesis. Turning to $\mathfrak T_2$, using the fact that $(f,\pi^k_T   z)_T =
(\pi^k_T   f, z)_T$ and exploiting the orthogonality property \eqref{orthogonality}, we have
$$\mathfrak T_2 = \sum_{T\in \cc T_h}  a_{| T}(\varpi^{k+2}_T u - u, \varpi^{k+2}_T z - z) + \sum_{T\in \cc T_h} \mr s_T(\wh{\ulr u}_T,\wh{\ulr z}_T)
+ (f-\pi^{k}_h f,z) \eqcolon \mathfrak T_{2,1} + \mathfrak T_{2,2} + \mathfrak T_{2,3}.$$
We have that $|\mathfrak T_{2,1}| \lesssim h^{k+3} \|u\|_{{H^{k+3}(\cc T_h)}}\|e_h\|$ by the Cauchy--Schwarz inequality, the approximation properties of $\varpi^{k+2}_T$, and biharmonic regularity.
An analogous bound can be obtained for $|\mathfrak T_{2,2}|$. Finally, we observe that $\mathfrak T_{2,3} = (f-\pi^{k}_h  f, z - \pi^{k}_h  z)$ by the definition \eqref{L2proj} of the $L^2$-orthogonal projector.
Using the approximation properties \eqref{stima} of $\pi^k_T$ with $l=k$, $m=0$, and $s=k+1$ for the first factor, $s=2$ for the second, we obtain
$$
|\mathfrak T_{2,3}| \le \| f-\pi^{k}_h f\| \| z-\pi^{k}_h z\| \lesssim h^{k+1} \|f\|_{H^{k+1}(\cc T_h)} h^2 \|z\|_{H^2(\Omega)}
\lesssim h^{k+3} \| f \|_{H^{k +1}(\cc T_h)} \|e_h\|.
$$
This concludes the proof.
\end{proof}
\subsection{Numerical examples}\label{sec:numerical.examples}
In this section we solve problem \eqref{static_hho} for $\bb A = \bb I$ (i.e., the biharmonic equation) in the unit square and, with a view towards testing the convergence of the method in the case of more complex geometries, in a L-shaped domain as well.
\subsubsection{Unit square}
In this first case, the domain under consideration is $\Omega = (0,1)\times(0,1)$. The right-hand side $f$ is set in agreement with the exact solution
$$u(x,y) = x^2(1-x)^2y^2(1-y)^2,$$
on three different meshes: triangular, cartesian and hexagonal (cf.~Fig.~\ref{meshes}). Figures \ref{test:err_en} and \ref{test:err_L2} show convergence results in the energy norm and in the $L^2$-norm, respectively, for different meshes and polynomial degrees, up to three. We consider $\|\wh{\ulr u}_h - \ulr u_h\|_{\mr{a},h}$ and $\|\pi^k_h u - u_h\|$ as measures of the error in the energy norm and in the $L^2$-norm, respectively. Since biharmonic regularity \eqref{biell_reg} is satisfied (the domain is convex and the exact solution is of class $C^\infty$), the numerical results show asymptotic convergence rates that match those predicted by the theory, i.e. estimates \eqref{errore} and \eqref{L2:err.est.lemma}, in all of the three cases.
\begin{figure}
\centering
\subfloat[Triangular]{\includegraphics[scale=1.2]{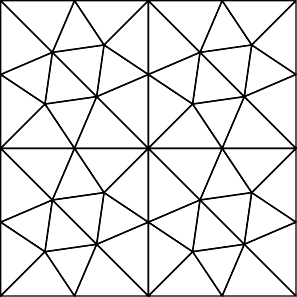}}
\qquad\quad\quad
\subfloat[Cartesian]{\includegraphics[scale=1.2]{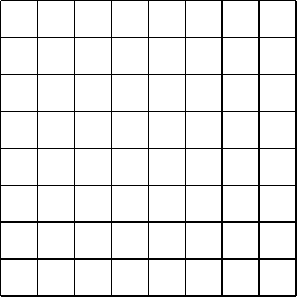}}
\qquad\quad\quad
\subfloat[Hexagonal]{\includegraphics[scale=1.2]{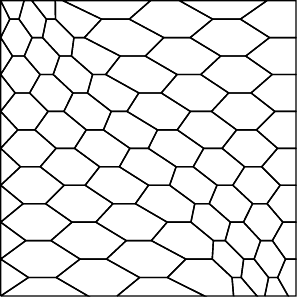}}
\caption{Meshes used for the numerical tests} 
\label{meshes}
\end{figure} 
\begin{figure}[h!]
\includegraphics{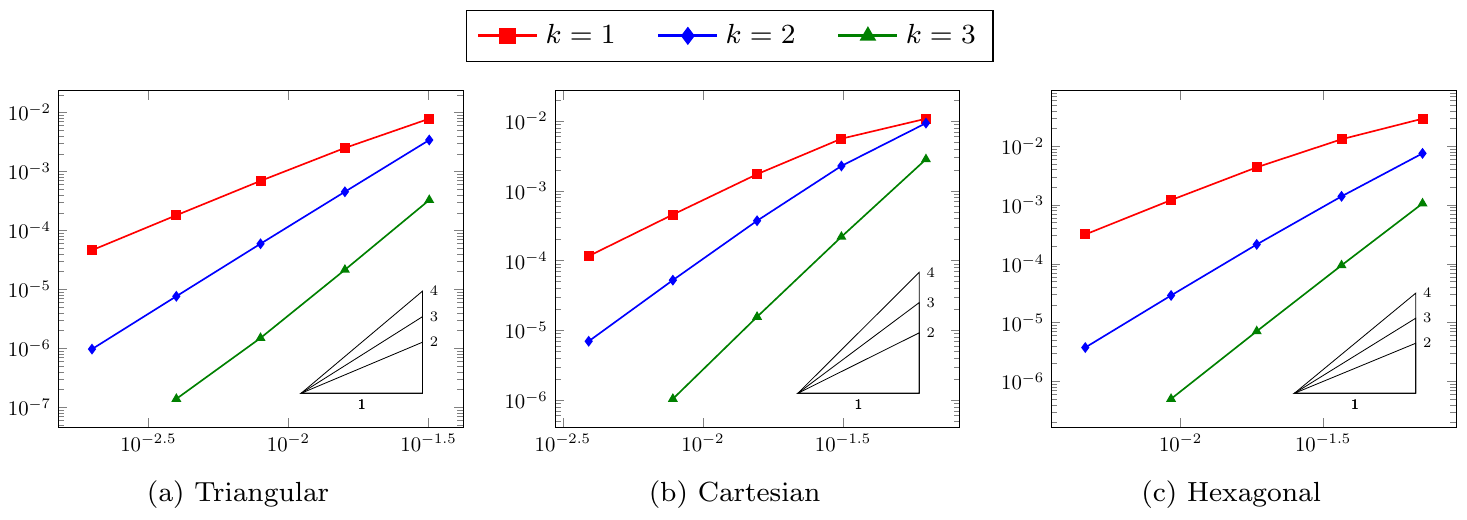}
\caption{$\|\wh{\ulr u}_h - \ulr u_h\|_{\mr{a},h}$ vs.~$h$ for three different meshes}
\label{test:err_en}
\end{figure}
\begin{figure}[h!]
\hspace{-.15cm}\includegraphics{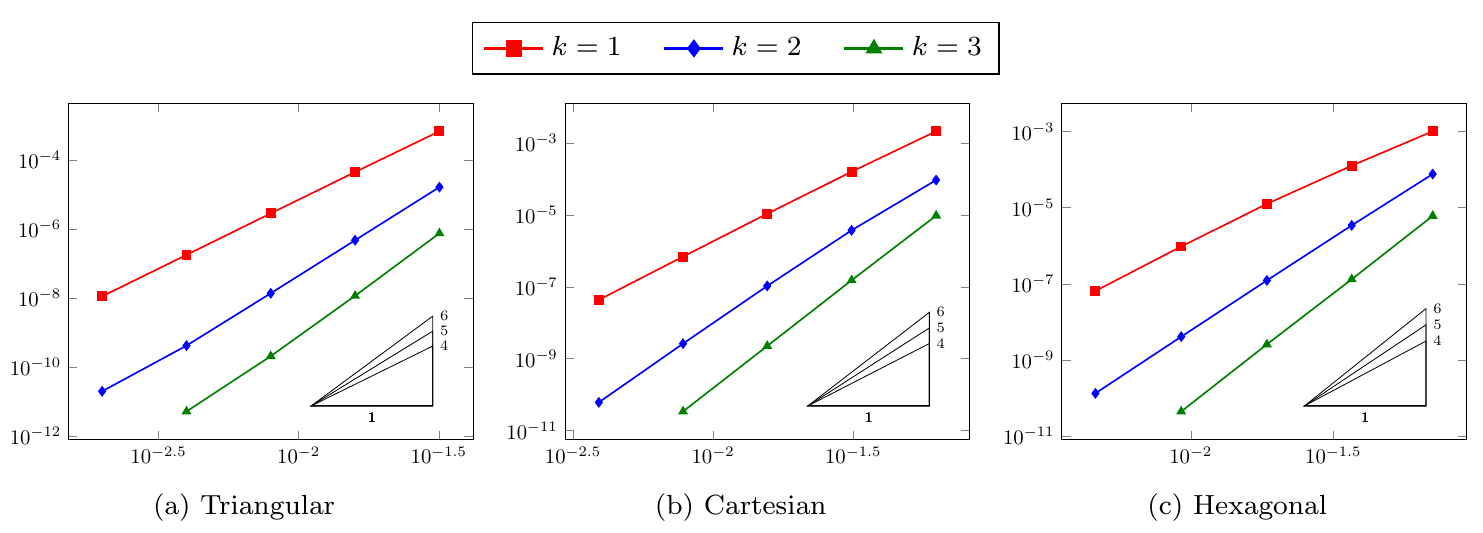}
\caption{$\|\pi^k_h u -  u_h\| $ vs.~$h$ for three different meshes}
\label{test:err_L2}
\end{figure}
{%
Also, we check the numerical convergence of the discrete energy \eqref{discr_en} with four uniformly refined triangular meshes, and a polynomial degree $k$ ranging from 1 to 4. As Table \ref{square_table} shows, only three refinements are necessary when $k\in\{2,3,4\}$ to achieve a five-significant-digit precision for the limit of the discrete energy.
\begin{table}
  \caption{Convergence of $\rm E(\ulr u_h)$ with four uniform mesh refinements for each polynomial degree \mbox{$k\in\{1,2,3,4\}$}. The number of triangular elements is given by $N$.}
  \label{square_table}{ 
    \begin{tabular}{|c|c|c|c|c|}
      \hline
      \diagbox[width=1.25cm, height=.55cm]{}{} & {$N = 56$} & {$N=224$} & {$N=896$} & {$N=3584$} \\ \hline
      $k=1$ & -1.662960e-03 & -1.635846e-03 & -1.632895e-03 & -1.632670e-03 \\ \hline
      $k=2$ & -1.637918e-03 & -1.632707e-03 & -1.632652e-03 & -1.632653e-03 \\ \hline
      $k=3$ & -1.632412e-03 & -1.632634e-03 & -1.632652e-03 & -1.632653e-03 \\ \hline
      $k=4$ & -1.632433e-03 & -1.632638e-03 & -1.632652e-03 &  -1.632653e-03 \\ \hline
    \end{tabular}
  }
\end{table}

}

We {finally} test the robustness of the variant of the HHO method based on the local bilinear form \eqref{forma_locale.eta} with respect to the user-dependent parameter $\eta$.
In Figures~\ref{eta_en} and~\ref{eta_l2} we plot, respectively, the energy- and $L^2$-norms of the error when $\eta$ varies from $10^{-3}$ to $10^3$ on fixed meshes corresponding to the third refinement level of the ones in Figure~\ref{meshes}.
From these plots, the robustness of the method can be appreciated, as the energy error spans only two orders of magnitude and the $L^2$-error spans four orders of magnitude, while the user-dependent parameter $\eta$ spans six orders of magnutide.
\begin{figure}[h!]
\includegraphics{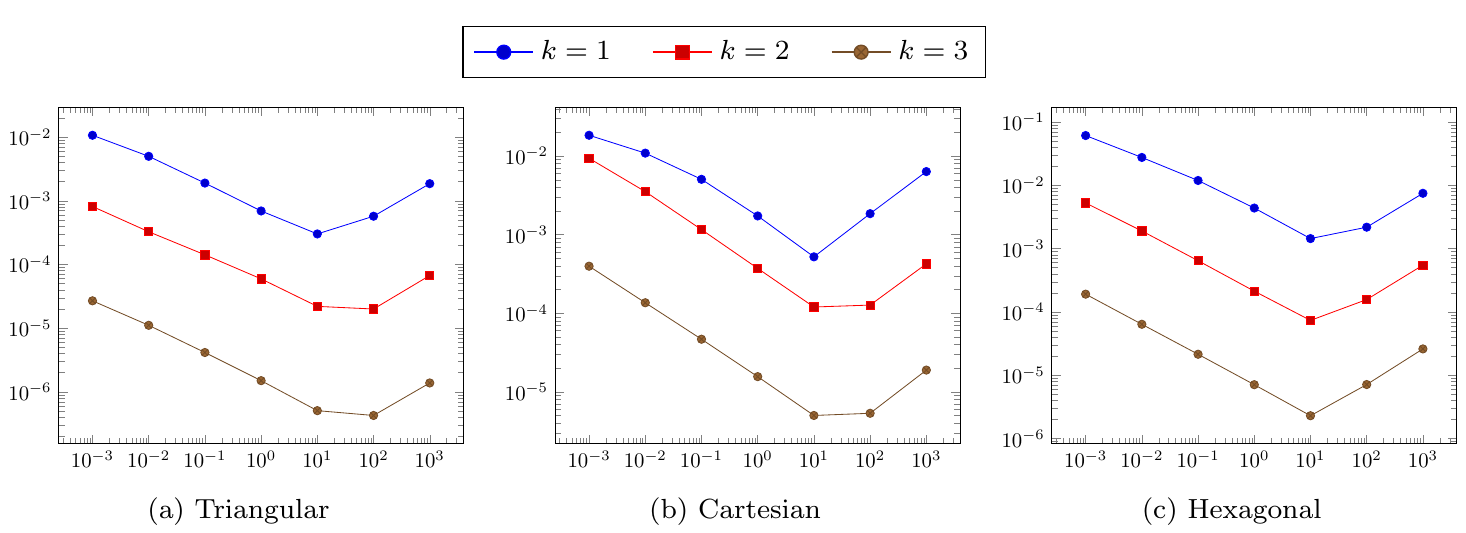}
\caption{$\|\wh{\ulr u}_h - \ulr u_h\|_{\mr{a},h}$ vs.~$\eta$ for the third refinement level of the Triangular, Cartesian, and Hexagonal meshes}
\label{eta_en}
\end{figure}
\begin{figure}[h!]
\includegraphics{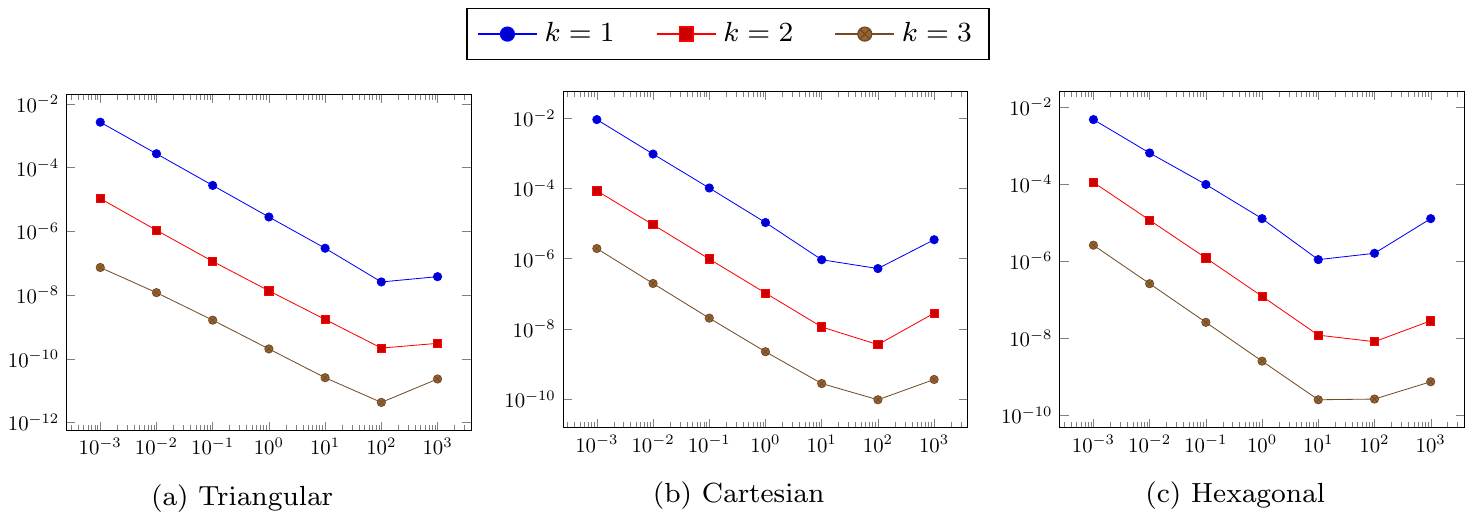}
\caption{$\|\pi^k_h u -  u_h\| $ vs.~$\eta$ for the third refinement level of the Triangular, Cartesian, and Hexagonal meshes}
\label{eta_l2}
\end{figure}
{%
\subsubsection{L-shaped domain}
We now consider the domain $\Omega = \left( (0,1) \times (0,1)\right) \setminus \left ( \left({1}/{2},1\right)\times\left({1}/{2},1\right)\right)$, and a uniform load \mbox{$f\equiv 1$}. Figure \ref{L_shaped} shows the numerical solution obtained for $k=3$ on five nested, uniformly refined triangular meshes. Since a closed-form solution is not available in this case, we check the numerical convergence of the discrete energy on the above-mentioned meshes, again for a polynomial degree $k$ ranging from $1$ to $4$. As Table \ref{L_shaped_table} shows, this energy converges numerically towards a value given by -2.80e-05 to two significant digits. This allows to conclude that the method converges even in situations where such singular geometries are considered. As expected, since biharmonic regularity is not satisfied in this case (because of the domain geometry), convergence is slower than in Table \ref{square_table}, and five mesh refinements are required to achieve a two-significant-digit precision for the limit. For further details, we refer the reader to Section \ref{sec:conclusioni:comp.cost.method}.
\begin{figure}
  \centering
  \subfloat[]{\includegraphics[keepaspectratio=true,scale=.15]{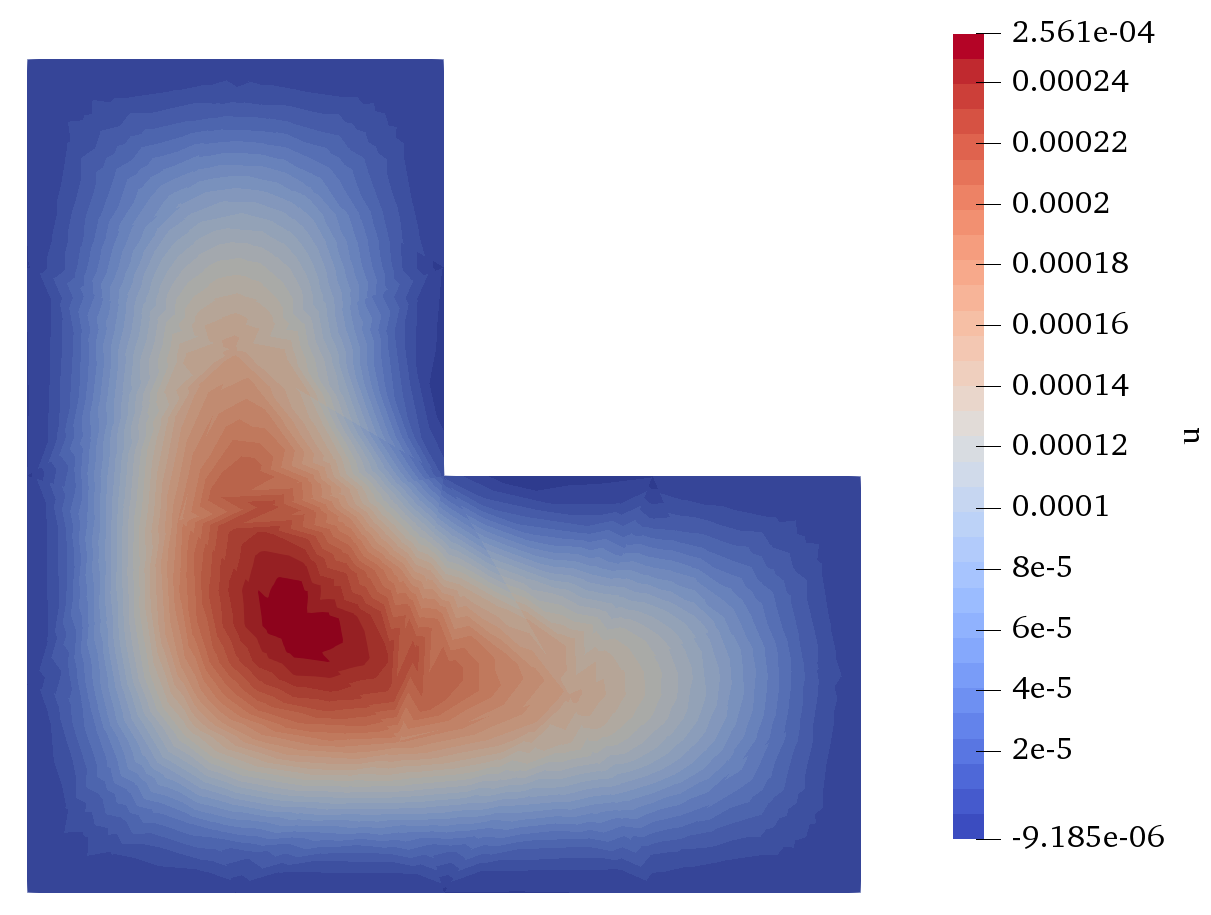}}
  \hspace{1cm}
  \raisebox{.75ex}{
    \subfloat[]{\includegraphics[keepaspectratio=true,scale=.15]{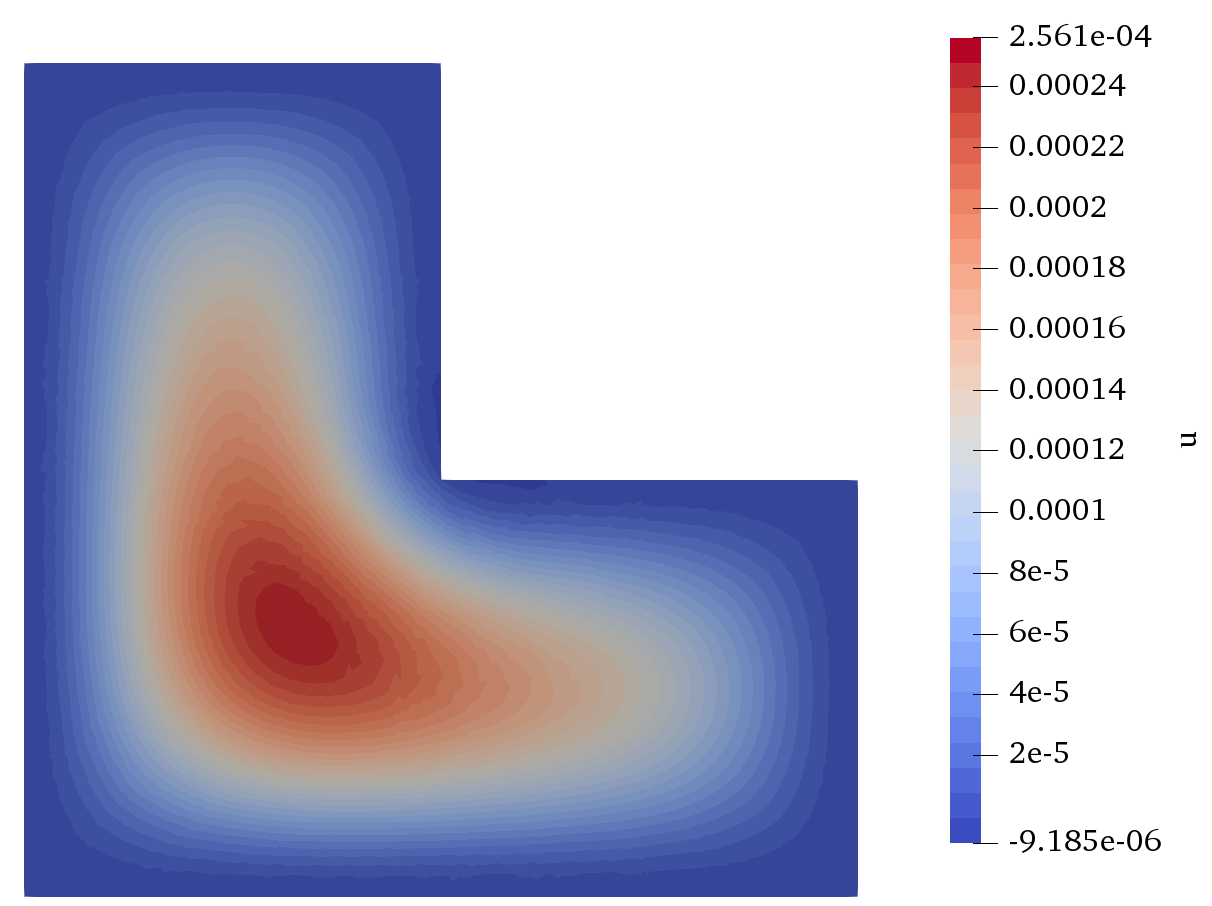}}}
  \hspace{1cm}
  \subfloat[]{\includegraphics[keepaspectratio=true,scale=.15]{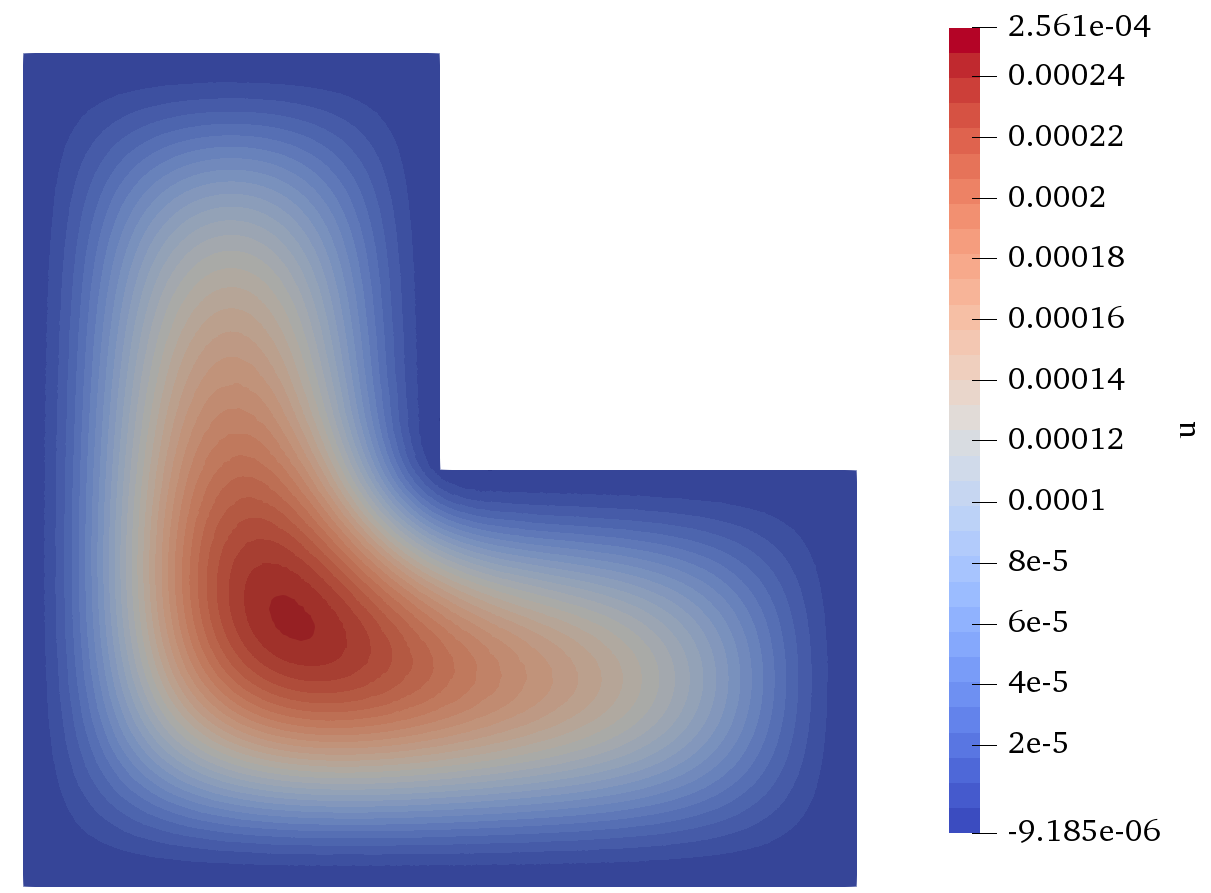}}
  \hspace{1cm}
  \subfloat[]{\includegraphics[keepaspectratio=true,scale=.15]{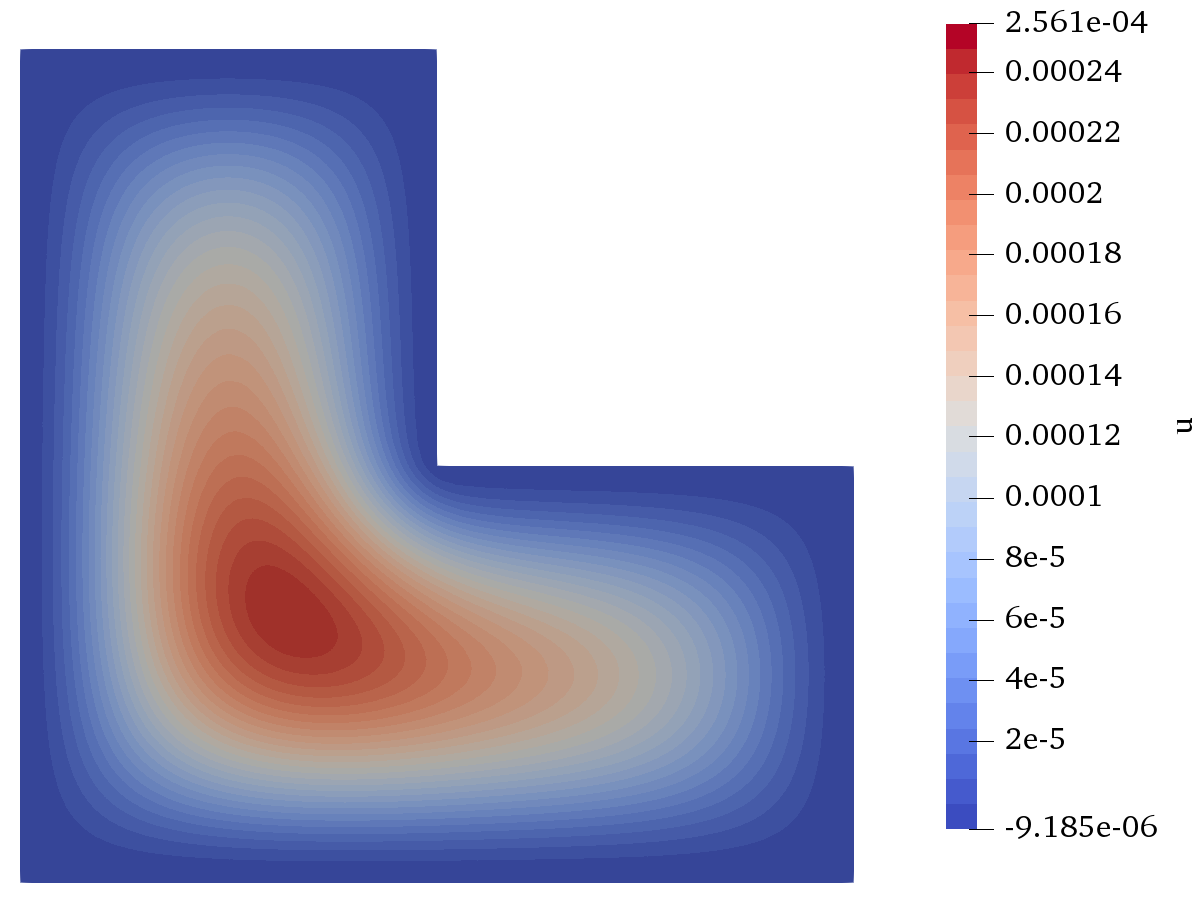}}
  \hspace{1cm}
  \subfloat[]{\includegraphics[keepaspectratio=true,scale=.15]{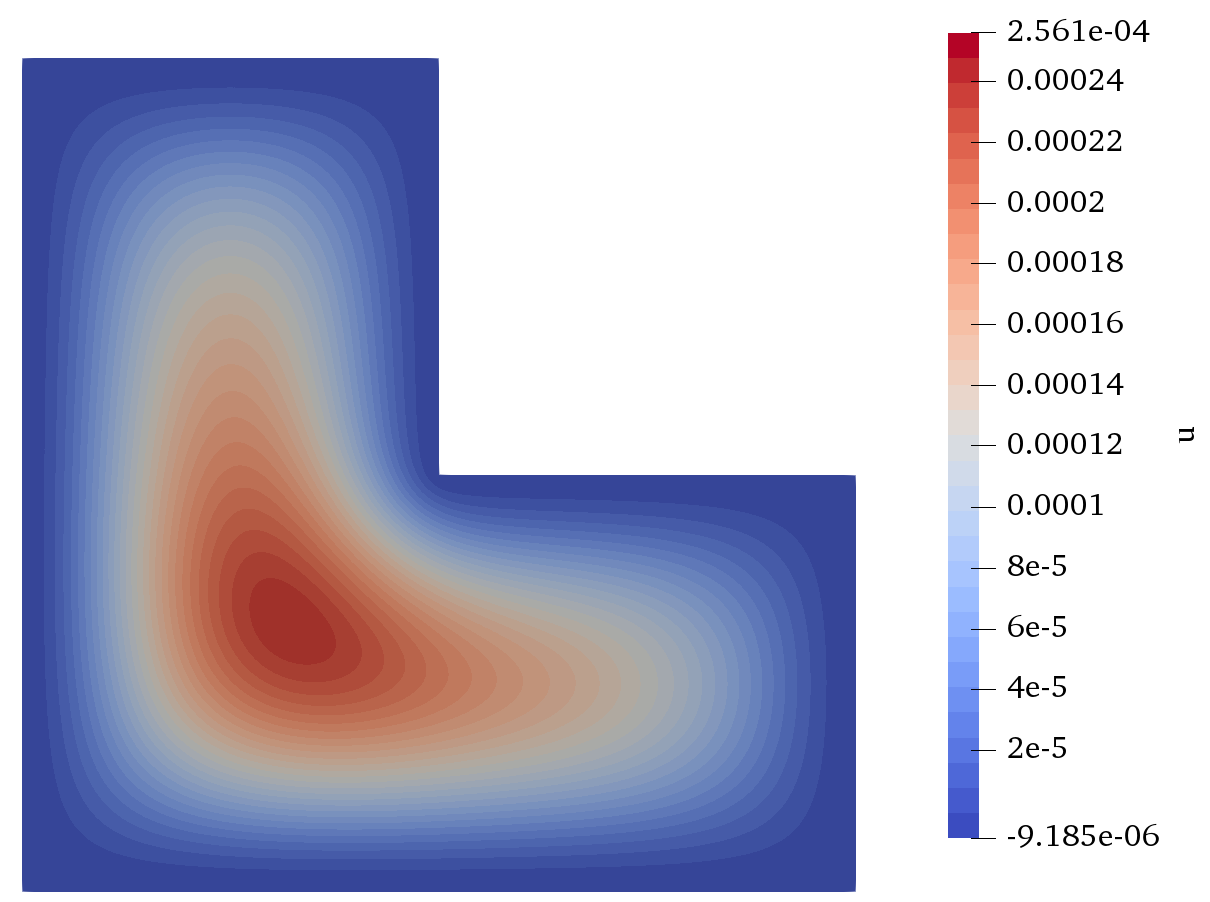}}
  \caption{Numerical solution obtained for a uniform load $f \equiv 1$, on five uniformly refined triangular meshes (with $N$ elements) of the domain, with $k=3$. Case (a): $N=34$; case (b): $N=136$; case (c): $N=544$; case (d): $N=2176$; case (e): $N=8704$.}
  \label{L_shaped}
\end{figure} 
\begin{table}
 \caption{Convergence of $\rm E(\ulr u_h)$ with uniform mesh refinements for each polynomial degree \mbox{$k\in\{1,2,3,4\}$}. The number of triangular elements is given by $N$.}
  \label{L_shaped_table}{ 
    \begin{tabular}{|c|c|c|c|c|c|}
      \hline
      \diagbox[width=1.25cm, height=.55cm]{}{} & {$N = 34$} & {$N=136$} & {$N=544$} & {$N=2176$} & {$N=8704$} \\ \hline
      $k=1$ & -4.208744e-05 & -3.071276e-05 & -2.885137e-05 & -2.833621e-05 & -2.813136e-05 \\ \hline
      $k=2$ & -3.167765e-05 & -2.945556e-05 & -2.858722e-05 & -2.824400e-05 & -2.809123e-05  \\ \hline
      $k=3$ & -2.944060e-05 & -2.868230e-05 & -2.828896e-05 & -2.811198e-05 & -2.803012e-05  \\ \hline
      $k=4$ & -2.899953e-05 & -2.845505e-05 & -2.818988e-05 &  -2.806669e-05 & -2.800896e-05 \\ \hline
    \end{tabular}
  }
\end{table}
}
%


\section{Local principle of virtual work and laws of action-reaction}\label{sec:acreac}

Let a mesh element $T\in\cc T_h$ be fixed.
At the continuous level, the deflection field $u$ satisfies, for all $v\in\bb P^k(T)$,
\begin{subequations}
\begin{equation}\label{cont_loc_eq}
  a_{| T}(u,v) + \sum_{F\in \cc F_T} (\bm M_T\bm n_{TF},\grad v)_{F} - \sum_{F\in \cc F_T} (\divb\bm M_T\cdot \bm n_{TF},v)_{F} = (f,v)_T,
\end{equation}
where $\bm M_T \coloneq -\bb A_T \grad^2 u$. Equation \eqref{cont_loc_eq} expresses the principle of \emph{virtual work} in the context
of Kirchhoff--Love plates, written for the mesh element element $T$ and with $\bb P^k(T)$ as the space of virtual deflections.
The quantities $\bm M_T \bm n_{TF}$ and $\divb\bm M_T \cdot \bm n_{TF}$ are internal actions and represent, respectively,
the moment and the (scalar) shear force exerted on the face $F \in \cc F_T$ by the adjacent
element. This can be viewed as a two-dimensional counterpart of Cauchy's hypothesis that the contact force density $\bm c$
at a point of an oriented surface $\Sigma$ in a three-dimensional continuum depends on $\Sigma$ only through the normal $\bm n$
to $\Sigma$ at that point; indeed, this implies that there is a second-order tensor field,
the Cauchy stress $\bm S$, such that, at each point of the three-dimensional body, $\bm c = \bm S \bm n$.

For an interface $F\in\cc F_{T_1}\cap\cc F_{T_2}$, with $T_1$, $T_2$ distinct elements of $\cc T_h$,
since $\bm n_{T_2 F} = - \bm n_{T_1 F}$,
both moments and shear forces obey the following \emph{laws of action-reaction}:
\begin{equation}\label{cont}
  \bm M_{T_1} \bm n_{T_1F} + \bm M_{T_2} \bm n_{T_2F} = \bm{0},\qquad
  \divb\bm M_{T_1} \cdot \bm n_{T_1F} + \divb\bm M_{T_2} \cdot \bm n_{T_2F} = 0.
\end{equation}
\end{subequations}
The denomination for equations \eqref{cont} emphasizes the fact that the moment (resp., shear force) exerted on element $T_1$ by element $T_2$ through the common interface $F$ is the opposite of the moment (resp., shear force) exerted on $T_2$ by $T_1$ through $F$.

We next show that the solution to discrete problem~\eqref{pb_discreto} satisfies discrete counterparts of~\eqref{cont_loc_eq} and~\eqref{cont}.
This requires a reformulation of the stabilization contribution in terms of the differences between face-based and element-based discrete unknowns.
Define the space
$$
\ulr D_{\partial T}^k \coloneq \left(\bigtimes_{F\in \cc F_T}\bb P^k(F)^2\right) \times \left(\bigtimes_{F\in\cc F_T} \bb P^{k}(F)\right)
$$
and the boundary difference operator $\ulr\delta_{\partial T}^k \colon \ulr U_T^k \to \ulr D_{\partial T}^k$ such that, for all $\ulr v_T\in\ulr U_T^k$,
$$
\ulr\delta_{\partial T}^k\ulr v_T
\equiv \big( (\bm{\delta}_{\grad,F}^k\ulr v_T)_{F\in\cc F_T}, (\delta_F^k\ulr v_T)_{F\in\cc F_T}\big)
\coloneq \big( (\bm v_{\grad,F}-\grad v_T)_{F\in \cc F_T}, (v_F - v_T)_{F\in \cc F_T}\big).
$$

\begin{proposition}[Boundary difference reformulation of $\mr s_T$]
  The local stabilization bilinear form $\mr s_T$ defined by \eqref{stabiliz} can be rewritten, for all $\ulr u_T,\ulr v_T\in\ulr U_T^k$,
  \begin{equation}\label{stabiliz'}
  \mr s_T(\ulr u_T, \ulr v_T) = \mr s_T((0,\ulr\delta_{\partial T}^k\ulr u_T),(0,\ulr\delta_{\partial T}^k\ulr v_T)).
  \end{equation}
\end{proposition}
\begin{proof}
As a consequence of~\eqref{eq:pTIT=varpi}, for all $v_T\in \bb P^{k}(T) \subset \bb P^{k+2}(T)$ it holds
\begin{equation}\label{pTvT}\pT \ulr I^k_T v_T = \varpi^{k+2}_T v_T = v_T,\end{equation}
where we have used the fact that, as a projector, $\varpi^{k+2}_T$ preserves polynomials up to degree $(k+2)$.
Now, using \eqref{pTvT} and the linearity of $\pT$, we have 
\begin{equation}\label{un}
  \pT \ulr v_T - v_T = \pT(\ulr v_T - \ulr I^k_T v_T)
  = \pT(0,\ulr{\delta}^k_{\de T}\ulr v_T).
\end{equation}
Also, for all $F\in \cc F_T$, it holds
\begin{equation}\label{deux}
\grad \pT\ulr v_T - \bm v_{\grad,F} = \grad( \pT\ulr v_T - \pT\ulr I^k_T v_T) - (\bm v_{\grad,F}-\grad v_T) = 
\grad \pT(0,\ulr{\delta}^k_{\de T} \ulr v_T) - \bm\delta^k_{\grad,F} \ulr v_T
\end{equation}
and, analogously,
\begin{equation}\label{trois}
\pT \ulr v_T - v_F = (\pT \ulr v_T - \pT \ulr I^k_T v_T) - (v_F - v_T) = \pT(0,\ulr{\delta}^k_{\de T}\ulr v_T) - \delta^k_F\ulr v_T.
\end{equation}
Using~\eqref{un},~\eqref{deux}, and~\eqref{trois} respectively in the first, second, and third term in the right-hand side of~\eqref{stabiliz}, the conclusion follows.
\end{proof}
Define now the residual operator
$$
\ulr R_{\de T}^k \equiv \left( (\bm R^k_{\grad, F})_{F\in\cc F_T},(R_F^{k})_{F\in\cc F_T}\right) \colon \ulr U^k_T \to \ulr D^k_{\de T}
$$
such that, for all $\ulr v_T \in \ulr U^k_T$ and all $\ulr\alpha_{\de T} \equiv ((\bm\alpha_{\grad,F})_{F\in\cc F_T}, (\alpha_F)_{F\in\cc F_T}) \in \ulr D^k_{\de T}$,
\begin{equation}\label{residuo}
(\ulr R^k_{\de T}\ulr v_T,\ulr\alpha_{\de T})_{0,\de T} \coloneq \! \!
\sum_{F\in \cc F_T} \!\! \Big( (\bm R^k_{\grad,F} \ulr v_T,\bm\alpha_{\grad,F})_F + (R^{k}_F\ulr v_T, \alpha_F)_F\Big) = 
\mr s_T((0,\ulr\delta_{\partial T}^k\ulr v_T),(0,\ulr\alpha_{\de T})).
\end{equation}
Problem~\eqref{residuo} is well-posed as a consequence of the Riesz representation theorem for the $L^2$-like product in the left-hand side.
\begin{lemma}[Local principle of virtual work and laws of action-reaction]
  Denote by $\ulr u_h\in \ulr U^k_{h,0}$ the unique solution to~\eqref{pb_discreto} and, for all $T\in\cc T_h$ and all $F\in\cc F_T$, define  the \emph{discrete moment and shear force}
\begin{equation}\label{fluxes}
\begin{aligned}
  \bm{\mathcal{M}}^k_{TF}(\ulr u_T) & \coloneq -\left( \bb (\bb A \grad^2 \pT \ulr u_T)\bm n_{TF} + \bm R^k_{\grad, F}\ulr u_T\right),\\
\mathcal{S}^k_{TF}(\ulr u_T) & \coloneq -\,\mr{\bf{div}}\,\bb A \grad^2 \pT \ulr u_T \cdot \bm n_{TF} + R^{k}_F\ulr u_T.
\end{aligned}
\end{equation}
Then, the following discrete counterparts of \eqref{cont_loc_eq} and \eqref{cont} hold, respectively:
For any mesh element $T\in\cc T_h$,
\begin{subequations}
  \begin{equation}\label{discr_loc_eq}
  a_{| T}(\pT \ulr u_T, v_T)\, + \sum_{F\in\cc F_T}  (\bm{\mathcal{M}}^k_{TF}(\ulr u_T), \grad v_T)_F \, - \sum_{F\in\cc F_T}  (\mathcal{S}^k_{TF}(\ulr u_T), v_T)_F = (f,v_T)_T,
  \quad\forall v_T \in \bb P^{k}(T),
  \end{equation}
  and, for any interface $F\in\cc F_{T_1}\cap\cc F_{T_2}$, with $T_1$, $T_2$ distinct elements of $\cc T_h$,
\begin{equation}\label{az_reaz}
  \bm{\mathcal{M}}^k_{T_1 F}(\ulr u_{T_1})+\bm{\mathcal{M}}^k_{T_2 F}(\ulr u_{T_2}) = \bm 0,\qquad
  \mathcal{S}^k_{T_1 F}(\ulr u_{T_1})+\mathcal{S}^k_{T_2 F}(\ulr u_{T_2}) = 0.
\end{equation}
\end{subequations}
\end{lemma}
\begin{proof}
Recalling the definition \eqref{forma_locale} of $\mr a_T$, and using the reformulation~\eqref{stabiliz'} of ${\mr s}_T$ together with the definition~\eqref{residuo} of the residual operator, it is inferred from the discrete problem~\eqref{pb_discreto} that, for all $\ulr v_h\in\ulr U_{h,0}^k$, it holds
\begin{equation}\label{discr_loc1}
  \sum_{T\in\cc T_h}\left(
  a_{| T}(\pT \ulr u_T, \pT \ulr v_T)
  + (\ulr R^k_{\de T}\ulr u_T, \ulr{\delta}^k_{\de T}\ulr v_T)_{0,\de T}
  \right) = (f,v_h).
\end{equation}
Using the definition~\eqref{reconstruction_locale2} of $\pT \ulr v_T$ with $w = \pT \ulr u_T$ for the first term, and recalling~\eqref{residuo} and \eqref{fluxes}, we can rewrite~\eqref{discr_loc1} as
\begin{equation}\label{discr_loc}
  \sum_{T\in\cc T_h} \! \! \left(
  a_{| T}(\pT \ulr u_T, v_T)
  - \! \! \sum_{F\in \cc F_T} (\bm{\mathcal{M}}^k_{TF}(\ulr u_T), \bm v_{\grad,F} - \grad v_T)_F
  +\! \! \sum_{F\in \cc F_T} (\mathcal{S}^k_{TF}(\ulr u_T), v_F - v_T)_F
  \right) = (f,v_h).
\end{equation}
Thus, for a given mesh element $T\in\cc T_h$, choosing in \eqref{discr_loc} $\ulr v_h$ such that $v_T$ spans $\bb P^{k}(T)$, $v_{T'}\equiv 0$ for all $T'\in\cc T_h\setminus\{T\}$, $\bm v_{\grad,F}\equiv\bm 0$ and $v_F\equiv 0$ for all $F\in\cc F_h$ immediately yields \eqref{discr_loc_eq}.
Next, for a given interface $F\in\cc F_{T_1}\cap\cc F_{T_2}$, choosing in \eqref{discr_loc} $\ulr v_h$ such that $v_T\equiv 0$ for all $T\in\cc T_h$, $\bm v_{\grad,F'}\equiv\bm 0$ for all $F'\in\cc F_h\setminus\{F\}$, $v_F\equiv 0$ for all $F\in\cc F_h$, and letting $\bm v_{\grad,F}$ span $\bb P^k(F)^2$ yields the first equation in~\eqref{az_reaz}.
Similarly, choosing in \eqref{discr_loc} $\ulr v_h$ such that $v_T\equiv 0$ for all $T\in\cc T_h$, $\bm v_{\grad,F}\equiv\bm 0$ for all $F\in\cc F_h\setminus\{F\}$, $v_{F'}\equiv 0$ for all $F\in\cc F_h\setminus\{F\}$, and letting $v_F$ span $\bb P^k(F)$ yields the second equation in~\eqref{az_reaz}.
\end{proof}

\section{Properties of the discrete bilinear form}\label{sec:analysis}

This section contains the proofs of the technical Lemmas~\ref{coerciv} and~\ref{stab_consist}.

\subsection{Local coercivity and boundedness}\label{sec:analysis:well-posedness}

\begin{proof}[Proof of Lemma~\ref{coerciv}]
Let a mesh element $T\in\cc T_h$ be fixed, and let $\ul{\mr v}_T\in\ul{\mr U}_T^k$. \\ \\
(i) \emph{Coercivity.}
Taking $ w=v_T \in \bb P^{k}(T)\subset\bb P^{k+2}(T) $ in (\ref{reconstruction_locale2}) gives
$$
\begin{aligned}
  a_{|T}( v_T,v_T)= a_{|T}(\pT \ul{\mr v}_T,v_T)
  + \!\! \sum_{F\in\cc F_T} \left(\bm v_{\grad,F}-\grad v_T,\bm M_{v_T}\bm n_{TF}\right)_F
  - \!\! \sum_{F\in \cc F_T} \big(v_F-v_T, \divb \bm M_{v_T} \cdot \bm n_{TF}\big)_F.
\end{aligned}
$$
Using the Cauchy--Schwarz inequality to bound the first term in the right-hand side, the Cauchy--Schwarz and discrete trace~\eqref{tr_discr} inequalities to bound the second, and the Cauchy--Schwarz, discrete trace~\eqref{tr_discr} and inverse~\eqref{inv} inequalities to bound the third, and simplifying we obtain:
\begin{equation}\label{primo pezzo}
  \|\bb A_T^{\nicefrac{1}{2}} \grad^2 v_T\|_T
  \lesssim
  \left(
  \|\bb A_T^{\nicefrac{1}{2}} \grad^2\pT \ul{\mr v}_T\|_T^2
  + \frac{\cc A_T^+}{h_T}\!\!\sum_{F \in \cc F_T}  \| \bm v_{\grad,F}-\grad v_T\|_F^2
  + \frac{\cc A_T^+}{h_T^3}\!\!\sum_{F\in \cc F_T} \| v_F-v_T\|_F^2
  \right)^{\nicefrac12}.
\end{equation}
It remains to estimate the boundary terms inside the parentheses using the $\|{\cdot}\|_{\mr{a},T}$-seminorm.

(i.a) \emph{Bound on $\frac{\cc A_T^+}{h_T}\sum_{F\in\cc F_T}\|\bm v_{\grad,F}-\grad v_T\|_F^2$.}
For all $F\in\cc F_T$, inserting $\pm\bm\pi_F^k\grad\left(\pT \ul{\mr v}_T - \pi^k_T\pT \ul{\mr v}_T\right)$ into the norm and using the linearity of $\bm\pi_F^k$ and the fact that it preserves polynomials in $\bb P^{k}(F)^2$ as a projector, we obtain
\begin{equation}\label{b}
\begin{aligned}
  &\|\bm v_{\grad,F} - \grad v_T \|_F
  \\
  &\quad=
  \|\bm\pi_F^k\left(\bm v_{\grad,F}-\grad \pT \ul{\mr v}_T\right)
  + \bm\pi_F^k\grad \left(\pT \ul{\mr v}_T - \pi_T^k\pT\ul{\mr v}_T\right)
  + \grad \left(\pi_T^k\pT\ul{\mr v}_T - v_T\right)
  \|_F
  \\
  &\quad\lesssim
  \| \bm\pi_F^k\left(\bm v_{\grad,F}-\grad  \pT \ul{\mr v}_T\right)\|_F
  {+} \|\bm\pi_F^k\grad( \pT \ul{\mr v}_T - \pi^k_T    \pT \ul{\mr v}_T)\|_F
  {+} \|\grad \pi^k_T  ( \pT \ul{\mr v}_T-v_T)\|_F
  \\
  &\quad\eqcolon \mathfrak T_1 + \mathfrak T_2 + \mathfrak T_3,
\end{aligned}
\end{equation}
where we have used the triangle inequality to pass to the second line.
By the definition~\eqref{stabiliz} of $\mr{s}_T$, we readily infer that
$$
h_T^{-\nicefrac{1}{2}} \sqrt{\cc A_T^+} |\mathfrak T_1| \lesssim
\|\ul{\mr v}_T\|_{\mr{a},T}.
$$
Using the $L^2(F)^2$-boundedness of $\bm\pi_F^k$ followed by the discrete trace inequality \eqref{tr_discr}, we can write $|\mathfrak T_2|\lesssim h_T^{-\nicefrac12}\|\grad( \pT \ul{\mr v}_T - \pi^k_T    \pT \ul{\mr v}_T)\|_T$.
Then, by the approximation properties \eqref{stima} of $\pi^k_T  $ with $l=k$, $m=1$, and $s=2$, we infer that
\begin{equation}\label{eq:stability:est.T2}
|\mathfrak T_2|
\lesssim h_T^{\nicefrac{1}{2}} |\pT\ul{\mr v}_T|_{H^2(T)}
\lesssim h_T^{\nicefrac{1}{2}} \|\grad^2 \pT\ul{\mr v}_T\|_T,
\end{equation}
so that
$$
h_T^{-\nicefrac{1}{2}}\sqrt{\cc A_T^+} |\mathfrak T_2|  \lesssim  \|\ul{\mr v}_T\|_{\mr{a},T}.
$$
{Notice, in passing,  that in the second bound in \eqref{eq:stability:est.T2} we have used the fact that $k \ge 1$.}
Finally, the third term in the right-hand side of~\eqref{b} can be estimated using the discrete trace \eqref{tr_discr} and inverse \eqref{inv} inequalities together with the definition~\eqref{stabiliz} of $\rm s_T$ as follows:
$$
h_T^{-\nicefrac{1}{2}}\sqrt{\cc A_T^+} |\mathfrak T_3|
\lesssim h_T^{-2}\sqrt{\cc A_T^+}\|\pi^k_T  ( \pT \ul{\mr v}_T-v_T)\|_T
\le \|\ul{\mr v}_T\|_{\mr{a},T}.
$$
Hence, multiplying \eqref{b} by $h_T^{-\nicefrac{1}{2}}\sqrt{\cc A^+_T}$, squaring, summing over $F\in\cc F_T$, using the above estimates for $\mathfrak T_1$, $\mathfrak T_2$, $\mathfrak T_3$, and recalling the uniform bound~\eqref{eq:bnd.faces} on $\card(\cc F_T)$, we have
\begin{equation}\label{b:1}
  \frac{\cc A_T^+}{h_T}\sum_{F\in\cc F_T}\|\bm v_{\grad,F}-\grad v_T\|_F^2\lesssim \|\ul{\mr v}_T\|_{\mr{a},T}^2.
\end{equation}

(i.b) \emph{Bound on $\frac{\cc A_T^+}{h_T^3}\sum_{F\in\cc F_T}\| v_F-v_T\|_F^2$.}
For all $F\in\cc F_T$, inserting $\pm\pi_F^k\left(\pT\ul{\mr v}_T - \pi_T^k\pT\ul{\mr v}_T\right)$ into the norm, and using the linearity of $\pi_F^k$ and $\pi_T^k$ together with the fact that they preserve polynomials up to degree $k$ as projectors, we have that
\begin{equation}\label{a}
  \begin{aligned}
    \|v_F-v_T\|_F &=
    \|\pi^k_F\left(v_F - \pT\ul{\mr v}_T\right)
    + \pi^k_F\left(\pT \ul{\mr v}_T - \pi^k_T   \pT \ul{\mr v}_T\right)
    + \pi^k_T(\pT \ul{\mr v}_T  - v_T )\|_F \\
    &\le
    \|\pi^k_F(v_F - \pT \ul{\mr v}_T)\|_F
    + \|\pi^k_F (\pT \ul{\mr v}_T - \pi^k_T   \pT \ul{\mr v}_T)\|_F
    + \| \pi^k_T(\pT \ul{\mr v}_T - v_T)\|_F
    \\
    &\eqcolon \mathfrak T_1 + \mathfrak T_2 + \mathfrak T_3.
  \end{aligned}
\end{equation}
By the definition~\eqref{stabiliz} of $\mr{s}_T$, it is readily inferred that
$$
h_T^{-\nicefrac{3}{2}} \sqrt{\cc A^+_T} |\mathfrak T_1| \lesssim\|\ul{\mr v}_T\|_{\mr{a},T}.
$$ 
The second term can be estimated as follows:
$$
|\mathfrak T_2| \lesssim h_T^{-\nicefrac{1}{2}} \| \pT \ul{\mr v}_T - \pi^k_T   \pT \ul{\mr v}_T\|_T
\lesssim h_T^{-\nicefrac{1}{2}}h_T^2
| \pT \ul{\mr v}_T |_{H^2(T)} \lesssim h_T^{\nicefrac{3}{2}} \|\grad^2 \pT \ul{\mr v}_T\|_T,
$$
where we have used the $L^2(F)$-boundedness of $\pi^k_F$, the discrete trace inequality \eqref{tr_discr}, the uniform equivalence of face and element diameters \eqref{meshreg} to replace $h_F$ with $h_T$, and the approximation property \eqref{stima} with $l=k$, $s=2$, and $m=0$. {Again, here the hypothesis $k \ge 1$ is necessary to infer the second bound.}
Hence,
$$
h_T^{-\nicefrac{3}{2}} \sqrt{\cc A_T^+} |\mathfrak T_2| \lesssim \|\ul{\mr v}_T\|_{\mr{a},T}.
$$
Finally, using the discrete trace inequality~\eqref{tr_discr} followed by the definition~\eqref{stabiliz} of $\mr{s}_T$, we have
$$
h_T^{-\nicefrac{3}{2}}\sqrt{\cc A_T^+}|\mathfrak T_3| \lesssim \|\ul{\mr v}_T\|_{\mr{a},T}.
$$
Multiplying \eqref{a} by
$h_T^{-\nicefrac{3}{2}}\sqrt{\cc A^+_T}$, squaring, summing over $F\in\cc F_T$, using the above estimates for $\mathfrak T_1$, $\mathfrak T_2$, $\mathfrak T_3$, and recalling the uniform bound~\eqref{eq:bnd.faces} on $\card(\cc F_T)$, we arrive at
\begin{equation}\label{a:1}
  \frac{\cc A_T^+}{h_T^3}\sum_{F\in\cc F_T}\|v_F-v_T\|_F^2 \lesssim \|\ul{\mr v}_T\|_{\mr{a},T}^2.
\end{equation}

(i.c) \emph{Conclusion.} Combining~\eqref{primo pezzo}, \eqref{a:1}, and~\eqref{b:1}, the first inequality in \eqref{equivnorm} follows.

(ii) \emph{Boundedness.} Taking $ w=\pT \ul{\mr v}_T $ in \eqref{reconstruction_locale2}, using the Cauchy--Schwarz, discrete trace \eqref{tr_discr} and inverse inequalities \eqref{inv}, and simplifying, we get
\begin{equation}\label{e}
  \| \bb A_T^{\nicefrac{1}{2}}\grad^2 \pT \ul{\mr v}_T \|_{T}\lesssim\|\ul{\mr v}_T\|_{\bb A, T},
\end{equation}
which bounds the portion of $\|\ul{\mr v}_T\|_{\mr{a},T}$ stemming from the consistency term in \eqref{forma_locale}.
It remains to bound on the local stabilization terms in $\mr{s}_T(\ul{\mr v}_T, \ul{\mr v}_T)$.

(ii.a) \emph{Bound on $\frac{\cc A_T^+}{h_T^4} \| \pi^k_T  (\pT\ul{\mr v}_T - v_T)\|_T^2$.}
Inserting $\pm\pT\ul{\mr v}_T$ into the norm and using the triangle inequality, we have that
\begin{equation}\label{f}
  \| \pi^k_T  (\pT\ul{\mr v}_T - v_T)\|_T
  \le \|\pi^k_T   \pT \ul{\mr v}_T - \pT \ul{\mr v}_T\|_T + \|\pT \ul{\mr v}_T - v_T\|_T
  \eqcolon \mathfrak T_1 + \mathfrak T_2.
\end{equation}
For the first term, using the approximation property \eqref{stima} with $l=k$, $m=0$, and $s=2$, and \eqref{e}, we get
$$
h_T^{-2}\sqrt{\cc A_T^+}|\mathfrak T_1| \lesssim  \|\ul{\mr v}_T\|_{\bb A,T}.
$$
{Once more, we use here the fact that $k \ge 1$.}
For the second term, inserting $0=-\pi_T^1\pT\ul{\mr v}_T+\pi_T^1 v_T$ into the norm (see \eqref{fermeture}) and using the triangle inequality, we obtain
\begin{equation*}
  |\mathfrak T_2|
  = \|\pT\ul{\mr v}_T - \pi^1_T \pT\ul{\mr v}_T + \pi^1_T v_T  - v_T \|_T \\
  \le \|\pT \ul{\mr v}_T - \pi^1_T\pT \ul{\mr v}_T\|_T  + \|\pi^1_T v_T - v_T\|_T.
\end{equation*}
The approximation property \eqref{stima} with $l=1$, $m=0$, and $s=2$ gives $\|\pT \ul{\mr v}_T - \pi^1_T\pT \ul{\mr v}_T\|_T \lesssim h_T^2 \|\grad^2\pT \ul{\mr v}_T\|_T$ and
$ \|v_T - \pi^1_T v_T\|_T \lesssim h_T^2 \|\grad^2 v_T\|_T$ so that, accounting for~\eqref{e},
$$
h_T^{-2} \sqrt{\cc A_T^+} |\mathfrak T_2| \lesssim \|\ul{\mr v}_T\|_{\bb A,T}.
$$
Squaring \eqref{f}, multiplying the resulting inequality by $\cc A_T^+ / h_T^4$, and using the above estimates for $\mathfrak T_1$ and $\mathfrak T_2$ together with the uniform bound~\eqref{eq:bnd.faces} on $\card(\cc F_T)$, we conclude that
$$
\frac{\cc A_T^+}{h_T^4} \| \pi^k_T  (\pT\ul{\mr v}_T - v_T)\|_T^2 \lesssim  \|\ul{\mr v}_T\|_{\bb A,T}^2.
$$

(ii.b) \emph{Bound on $\frac{\cc A_T^+}{h_T}\sum_{F\in\cc F_T}\|\bm \pi^k_F(\grad \pT\ul{\mr v}_T - \bm v_{\grad,F})\|_F^2$.}
For any $F\in\cc F_T$, inserting $\pm\grad v_T$ into the norm, invoking the linearity of $\bm\pi_F^k$ together with the fact that it preserves polynomials in $\bb P^{k}(F)^2$ as a projector, and using the triangle inequality, we have that
\begin{equation}\label{h}
  \begin{aligned}
    \|\bm\pi^k_F(\grad \pT\ul{\mr v}_T - \bm v_{\grad,F})\|_F
    &\le \| \bm\pi^k_F\grad\left(\pT \ul{\mr v}_T - v_T\right)\|_F + \|\grad v_T - \bm v_{\grad,F}\|_F \\
    &\lesssim h_T^{-\nicefrac{3}{2}} \| \pT\ul{\mr v}_T - v_T\|_T + \| \grad v_T - \bm v_{\grad,F}\|_F \\
    &\lesssim \frac{h_T^{\nicefrac12}}{\sqrt{\cc A_T^+}} \| \ul{\mr v}_T\|_{\bb A,T} + \| \grad v_T - \bm v_{\grad,F}\|_F
  \end{aligned}
\end{equation}
where to pass to the second line we have used the $L^2(F)^2$-boundedness of $\bm\pi_F^k$, the discrete trace inequality \eqref{tr_discr}, and the inverse inequality \eqref{inv}, while to pass to the third line we have estimated the first addend as the term $\mathfrak T_2$ in~\eqref{f} {(which requires again $k\ge1$)}.
Thus, squaring the above inequality, summing over $F\in\cc F_T$, multiplying it by $\cc A_T^+/h_T$, and using the uniform bound \eqref{eq:bnd.faces} on $\card(\cc F_T)$, we finally infer
\begin{equation}\label{eq:cont.bnd.2}
  \frac{\cc A_T^+}{h_T}\sum_{F\in\cc F_T}\|\bm \pi^k_F(\grad \pT\ul{\mr v}_T - \bm v_{\grad,F})\|_F^2
  \lesssim \|\ul{\mr v}_T\|_{\bb A,T}^2.
\end{equation}

(ii.c) \emph{Bound on $\frac{\cc A_T^+}{h_T^3}\sum_{F\in\cc F_T}\|\pi^k_F(\pT \ul{\mr v}_T - v_F) \|_F^2$.}
For any $F\in\cc F_T$, inserting $\pm v_T$ into the norm, invoking the linearity of $\pi_F^k$ together with the fact that it preserves polynomials in $\bb P^k(F)$ as a projector, and using the triangle inequality, we infer that
\begin{equation}\label{g}
  \begin{aligned}
    \|\pi^k_F(\pT \ul{\mr v}_T - v_F) \|_F
    &\le \|\pi^k_F(\pT \ul{\mr v}_T - v_T)\|_F + \|v_F-v_T\|_F \\
    &\lesssim h_T^{-\nicefrac{1}{2}} \|\pT \ul{\mr v}_T - v_T \|_T + \|v_F-v_T\|_F \\
    &\lesssim \frac{h_T^{\nicefrac32}}{\sqrt{\cc A_T^+}} \| \ul{\mr v}_T\|_{\bb A,T} + \|v_F-v_T\|_F,
  \end{aligned}
\end{equation}
where to pass to the second line we have used the $L^2(F)$-boundedness of $\pi^k_F$ followed by the discrete trace inequality \eqref{tr_discr} and the uniform equivalence of the element and face diameters expressed by~\eqref{meshreg}, while to pass to the third line we have estimated the first addend as the term $\mathfrak T_2$ in~\eqref{f} {and, once more, we used the fact that $k\ge 1$}.
Hence, multiplying \eqref{g} by $h_T^{-\nicefrac{3}{2}}\sqrt{\cc A_T^+}$, squaring, summing over $F\in\cc F_T$, recalling \eqref{f}, and using the uniform bound \eqref{eq:bnd.faces} on $\card(\cc F_T)$, we conclude that
\begin{equation}\label{eq:cont.bnd.1}
  \frac{\cc A_T^+}{h_T^3}\sum_{F\in\cc F_T}\|\pi^k_F(\pT \ul{\mr v}_T - v_F) \|_F^2 \lesssim \| \ul{\mr v}_T\|_{\bb A,T}^2.
\end{equation}

(ii.d) \emph{Conclusion.}
The second inequality in \eqref{equivnorm} then follows combining \eqref{e},~\eqref{eq:cont.bnd.1}, and~\eqref{eq:cont.bnd.2} and recalling the definition~\eqref{localseminorm} of $\|{\cdot}\|_{\bb A,T}$.
\end{proof} 


\subsection{Global coercivity, boundedness, and consistency}\label{sec:analysis:err.est}

\begin{proof}[Proof of Lemma~\ref{stab_consist}]
  (i) \emph{Coercivity and boundedness.} The norm equivalence~\eqref{global_stab} is an immediate consequence of Lemma~\ref{coerciv} together with the definition~\eqref{norm discrete} of the $\|{\cdot}\|_{\bb A,h}$-norm.

  (ii) \emph{Consistency.} Let us prove~\eqref{consistenza}. An element-wise integration by parts yields
\begin{equation}\label{uno}
\begin{alignedat}{2}
\hspace{-.3cm}(\div\divb \bb A \grad^2 v, w_h)  = & \sum_{T\in\cc T_h} \bigg(
(\bb A_T\grad^2 v,\grad^2 w_T)_T - \sum_{F\in \cc F_T} (\divb \bb A_T\grad^2 v\cdot\bm n_{TF},w_F-w_T)_F\\
 & + \sum_{F\in \cc F_T} \left( (\bb A_T\grad^2 v)\bm n_{TF},  \bm w_{\grad,F}-\grad w_T \right)_F\bigg),
\end{alignedat}
\end{equation}
where we have used the fact that moments and Kirchhoff shear forces are continuous at interfaces owing to the regularity of $v$ (see~\eqref{cont} for the expression of these continuity properties for the exact solution $u$) and that
homogeneous boundary conditions are embedded in $\ul{\mr U}^k_{h,0}$. Now, let
\begin{equation}\label{notazioni}
\widehat{\ul{\mr v}}_h \coloneq \ul{\mr I}^k_h v,\quad \widehat{\ul{\mr v}}_T \coloneq \ul{\mr I}^k_T(v_{| T})\quad
\hbox{and}\quad\widecheck v_T \coloneq \pT \widehat{\ulr v}_T = \varpi_T^{k+2} v_{|T}; 
\end{equation}
we have
\begin{equation}\label{due}
\begin{alignedat}{2}
\mr{a}_h(\widehat{\ul{\mr v}}_h,\ul{\mr w}_h) = &
 \sum_{T\in \cc T_h} \bigg( (\bb A_T\grad^2 \widecheck v_T, \grad^2 w_T)_T 
 -\sum_{F\in \cc F_T} \left(\divb \bb A_T\grad^2 \widecheck v_T\cdot \bm n_{TF},w_F-w_T\right)_F \\
& + \sum_{F\in \cc F_T} \left((\bb A_T\grad^2\widecheck v_T)\bm n_{TF}, \bm w_{\grad,F} - \grad w_T\right)_F
+ \mr{s}_T(\widehat{\ul{\mr v}}_T,\ul{\mr w}_T) \bigg).
\end{alignedat}
\end{equation}
Thus, letting $\cc E_h(\ul{\mr w}_h) \coloneq(\div\divb \bb A \grad^2 v, w_h) - \mr{a}_h(\widehat{\ul{\mr v}}_h,\ul{\mr w}_h)$,
\eqref{uno} and \eqref{due} yield
\begin{equation*}
\begin{alignedat}{2}
\cc E_h(\ul{\mr w}_h)  = & \sum_{T\in \cc T_h} \bigg( (\bb A_T\grad^2(\widecheck v_T - v),\grad^2 w_T)_T + 
\left( (\bb A_T\grad^2(\widecheck v_T - v))\bm n_{TF},\bm w_{\grad,F}-\grad w_T\right)_F \\
&- \sum_{F\in \cc F_T} (\divb \bb A_T\grad^2(\widecheck v_T - v)\cdot \bm n_{TF},w_F-w_T)_F\bigg)
+ \sum_{T\in \cc T_h} \mr{s}_T(\widehat{\ul{\mr v}}_T,\ul{\mr w}_T) \\
 \eqcolon &~ \mathfrak T_1 + \mathfrak T_2 + \mathfrak T_3 + \mathfrak T_4.
\end{alignedat}
\end{equation*}
By the definition \eqref{eq:biell} of the {local energy projector}, we have that
\begin{equation}\label{eq:consistenza:stima.T1}
  \mathfrak T_1 = 0.
\end{equation}
Using the approximation properties \eqref{approx_biell} with $l=k+2$, $s = k+3$, and $m=2,3$, we infer that
\begin{equation}\label{prima}
| \mathfrak T_2 + \mathfrak T_3| \lesssim 
h^{k+1} |v|_{{H^{k+3}(\cc T_h)}} \|\ul{\mr w}_h\|_{\bb A,h}.
\end{equation}
Moreover, for all $T\in\cc T_h$, we have $\mr{s}_T(\widehat{\ul{\mr v}}_T, \ul{\mr w}_T) \le \mr{s}_T(\widehat{\ul{\mr v}}_T,\widehat{\ul{\mr v}}_T)^{\nicefrac{1}{2}} \mr{s}_T(\ul{\mr w}_T \ul{\mr w}_T)^{\nicefrac{1}{2}}$;
as for the first factor, by \eqref{consistenza_s_T} we have
$\mr{s}_T(\widehat{\ul{\mr v}}_T,\widehat{\ul{\mr v}}_T)^{\nicefrac{1}{2}} \lesssim \sqrt{\cc A_T^+} h_T^{k+1} |v|_{H^{k+3}(T)},$ whereas the second inequality in \eqref{equivnorm} gives $\mr{s}_T(\ul{\mr w}_T, \ul{\mr w}_T)^{\nicefrac{1}{2}} \lesssim \|\ul{\mr w}_T\|_{\bb A,T}$, so that
\begin{equation}\label{seconda}
  |\mathfrak T_4| \lesssim  h^{k+1} |v|_{{H^{k+3}(\cc T_h)}}\|\ul{\mr w}_h\|_{\bb A,h}.
\end{equation}
Using \eqref{eq:consistenza:stima.T1}, \eqref{prima}, and \eqref{seconda} to estimate $\cc E_h(\ul{\mr w}_h)$, and using the resulting bound in the supremum in \eqref{consistenza} concludes the proof.
\end{proof}



\section{Proof of Theorem~\ref{thm:approx.biell}}\label{sec:approx.biell}

  (i) \emph{Proof of \eqref{approx_properties}}.
  We apply \cite[Lemma 3]{Di-Pietro.Droniou:17}. Therefore, proving \eqref{approx_properties}
amounts to proving the following estimates:
\begin{subequations}
\begin{align}
\|\grad^2 \varpi^l_T v\|_T & \lesssim  \|\grad^2 v\|_T, \label{m=2}\\
\|\grad \varpi^l_T v\|_T & \lesssim \left( \|\grad v\|_T + h_T \|\grad^2 v\|_T \right),\label{m=1}\\
\|\varpi^l_T v\|_T & \lesssim  \left( \|v\|_T + h_T \|\grad v\|_T + h_T^2 \|\grad^2 v\|_T \right). \label{m=0}
\end{align}
\end{subequations}
where $a \lesssim b$ means $a \le Cb$ with $C>0$ as in \eqref{approx_biell}.

We start by proving \eqref{m=2}. The definition \eqref{eq:biell} of $\varpi^l_T$ implies that
\begin{equation}\label{inf_stima}
\begin{aligned}
  \sqrt{\cc A_T^-}\, \|\grad^2( v- \varpi^l_T v)\|_T
  &\le \| \bb A_T^{\nicefrac{1}{2}} \grad^2(v - \varpi^l_T v)\|_T \\
  &\le \| \bb A_T^{\nicefrac{1}{2}} \grad^2(v - \pi^l_T v)\|_T 
  \lesssim\sqrt{\cc A_T^+}\|\grad^2 v\|_T,
\end{aligned}
\end{equation}
where we have used the definition of $\cc A_T^-$ (see Section~\ref{sec:discset:mesh}) in the first line, the characterization of $\varpi^l_T$ as $\arg\min_{z\in\bb P^l(T)}\| \bb A_T^{\nicefrac12}\grad^2(v-z)\|_T$ in the second line, along with the definition of $\cc A_T^+$ and the $H^2$-stability of the $L^2$-orthogonal projector (resulting from \eqref{stima} with $s=m=2$) to conclude.
Thus, using again the triangle inequality, we have that
$$
\|\grad^2 \varpi^l_T v\|_T \le \|\grad^2(\varpi^l_T v - v)\|_T + \|\grad^2 v\|_T \lesssim  \|\grad^2 v\|_T,
$$
and \eqref{m=2} is proved.

To prove \eqref{m=1}, we introduce the quantities $0=-\grad\pi^1_T\varpi^l_T v + \grad\pi^1_T v$ (recall the second condition in \eqref{eq:biell}) and $\pm\grad v$ inside the $L^2(T)$-norm of $\grad\varpi^l_T v$ to infer that
\begin{equation*}
\begin{aligned}
\|\grad\varpi^l_T v\|_T 
& \le \|\grad(\varpi^l_T v - \pi^1_T \varpi^l_T v)\|_T + \|\grad(v-\pi^1_T v)\|_T + \|\grad v\|_T \\
& \lesssim h_T \|\grad^2\varpi^l_T v\|_T + h_T \|\grad^2 v\|_T + \|\grad v\|_T \lesssim \left( \|\grad v\|_T + h_T \|\grad^2 v\|_T \right),
\end{aligned}
\end{equation*}
where we have used the approximation estimate \eqref{stima} for $\pi^1_T$ with $m=1$ and $s=2$ together with the fact that, for any $w\in H^2(T)$, $|w|_{H^2(T)}\lesssim\|\grad^2w\|_T$ to estimate the first two terms, and~\eqref{m=2} to conclude.

The proof of \eqref{m=0} is completely analogous. We obtain
\begin{equation*}
\begin{aligned}
\|\varpi^l_T v\|_T & \le \|\varpi^l_T v - \pi^1_T\varpi^l_T v\|_T + \|v - \pi^1_T v\|_T + \|v\|_T \\
& \lesssim h_T^2 \|\grad^2 \varpi^l_T v\|_T+ h_T \|\grad v\|_T + \|v\|_T \\
& \lesssim  \left( h_T^2 \|\grad^2 v\|_T + h_T \|\grad v\|_T +  \|v\|_T\right),
 \end{aligned}
\end{equation*}
where we have used \eqref{stima} to estimate the first two addends in the first line, with $m=0$ and $s=2$ for the first one
and with $m=0$ and $s=1$ for the second one. This concludes the proof of \eqref{approx_properties}. \\ \\

(ii) \emph{Proof of \eqref{approx_properties_b}}. For $m\le s-1$, by applying the continuous trace inequality \eqref{tr_cont} to
$w=\de^{\bm\alpha}(v - \varpi^l_T v) \in H^1(T)$ for all $\bm \alpha \in \bb N^2$ such that $\alpha_1+\alpha_2 = m$, we have
$$h_T^{\nicefrac{1}{2}} |v - \varpi^l_T v|_{H^m(\de T)} \lesssim |v - \varpi^l_T v|_{H^m(T)} + h_T |v - \varpi^l_T v|_{H^{m+1}(T)}.$$
The conclusion follows using \eqref{approx_properties} for $m$ and $m+1$ to bound the terms in the right-hand side.


\section{Concluding remarks}\label{sec:conclusioni}

Some concluding remarks are in order.

{%
\subsection{Computational cost of the method}\label{sec:conclusioni:comp.cost.method}
It is worth to draw some conclusions from the numerical tests set forth in Section \ref{sec:numerical.examples}, with particular reference to the L-shaped domain case. Indeed, in many applications, as well as from a theoretical viewpoint, it is interesting to estimate the computational cost of a given numerical method. Here, we can evaluate the computational cost of our method by comparing the sizes of the matrices associated with the bilinear form $\mr a_h$, as well as the number of nonzero elements of such matrices\footnote{
The latter, in particular, gives an insight into the stencil of the method.}, upon varying the polynomial degree $k$ and the number of elements $N$. These two quantities are represented in Tables \ref{taillematrice_table} and \ref{nonzeros_table}, respectively. As Table \ref{taillematrice_table} shows, in certain cases (compare, for instance, the results given by the two choices $k=2$, $N=136$ and $k=3$, $N=34$ in Table \ref{L_shaped_table}) using polynomials of high order on coarse triangulations is more convenient than using polynomials of lower order on finer triangulations to obtain a given numerical value of the discrete energy to two significant digits.
\begin{table}
 \caption{Size of the matrix associated with $\mr a_h$ with uniform mesh refinements for each polynomial degree \mbox{$k\in\{1,2,3,4\}$}. The number of triangular elements is given by $N$.}
  \label{taillematrice_table}{ 
    \begin{tabular}{|c|c|c|c|c|c|}
      \hline
      \diagbox[width=1.25cm, height=.55cm]{}{} & {$N = 34$} & {$N=136$} & {$N=544$} & {$N=2176$} & {$N=8704$} \\ \hline
      $k=1$ & 354 & 1320 & 5088 & 19968 & 79104 \\ \hline
      $k=2$ & 531 & 1980 & 7632 & 29952 & 118656  \\ \hline
      $k=3$ & 708 &  2640 & 10176 & 39936 & 158208  \\ \hline
      $k=4$ & 885 & 3300 & 12720 &  49920 & 197760 \\ \hline
    \end{tabular}
  }
\end{table}
\begin{table}
 \caption{Number of nonzero elements of the matrix associated with $\mr a_h$ with uniform mesh refinements for each polynomial degree \mbox{$k\in\{1,2,3,4\}$}. The number of triangular elements is given by $N$.}
  \label{nonzeros_table}{ 
    \begin{tabular}{|c|c|c|c|c|c|}
      \hline
      \diagbox[width=1.25cm, height=.55cm]{}{} & {$N = 34$} & {$N=136$} & {$N=544$} & {$N=2176$} & {$N=8704$} \\ \hline
      $k=1$ & 9468 & 37296 & 148032 & 589824 & 2354688 \\ \hline
      $k=2$ & 21303 & 83916 & 333072 & 1327104 & 5298048  \\ \hline
      $k=3$ & 37872 & 149184 & 592128 & 2359296 & 9418752  \\ \hline
      $k=4$ & 59175 & 233100 & 925200 &  3686400 & 14716800 \\ \hline
    \end{tabular}
  }
\end{table}
}
\subsection{Mixed formulations}

The results of this paper concern the primal formulation \eqref{static_hho} of the Kirchhoff--Love plate bending model problem. As it is well known, this problem admits  dual and mixed formulations that have been the basis for the development of mixed and hybrid nonconforming finite elements (see, e.g., \cite{brezzifortin}).
A HHO discretization based on a mixed formulation will make the object of a future work, as well as the study of its relation with the method presented here and its variations.
We notice, in passing, that a similar study for a second-order elliptic problem has been carried out in \cite{Aghili.Boyaval.ea:15} {and, in a more general setting, in \cite{Boffi.Di-Pietro:17}. The latter works can be regarded as a generalization to new generation polytopal methods of the classical hybridization techniques of Arnold--Brezzi \cite{Arnold.Brezzi:85}.}

\tbf{Acknowledgements.} The authors are grateful to Franco Brezzi (IMATI Pavia) for the fruitful discussions that have helped shape up this work.


\end{document}